\documentclass[a4paper,10pt]{article}

\usepackage{amscd,amssymb,amsmath,graphicx,verbatim}
\usepackage[dvips]{hyperref}
\usepackage{textcomp}
\usepackage[OT1,T1]{fontenc}

\usepackage[dvips,ps,all,graph,curve]{xypic}
\makeatletter
 \def\setlabelmargin#1{\labelmargin@=#1\relax }
\makeatother

\newcommand{\qed}{\nobreak \ifvmode \relax \else
      \ifdim\lastskip<1.5em \hskip-\lastskip
      \hskip1.5em plus0em minus0.5em \fi \nobreak
      \vrule height0.75em width0.5em depth0.25em\fi}

\newcounter{Definitioncount}
\newtheorem{theorem}{Theorem}

\newtheorem{proposition}[theorem]{Proposition}
\newtheorem{corollary}[theorem]{Corollary}
\newenvironment{example}[1][Example]{\begin{trivlist}
\item[\hskip \labelsep {\bfseries #1}] }{\end{trivlist}}
\newenvironment{proof}[1][Proof]{\begin{trivlist}
\item[\hskip \labelsep {\bfseries #1}]}{\end{trivlist}}
\newenvironment{definition}[1][Definition]{\begin{trivlist}
\item[\hskip \labelsep {\bfseries #1}] \refstepcounter{Definitioncount} \textbf{\arabic{Definitioncount}} }{\end{trivlist}}
\newenvironment{remark}[1][Remark]{\begin{trivlist}
\item[\hskip \labelsep {\bfseries #1}]}{\end{trivlist}}

\newcommand{\bm}{{\cal{M}}}
\newcommand{\bid}{{\cal{I}}}
\newcommand{\bc}{{\cal{C}}}
\newcommand{\ba}{{\cal{A}}}
\newcommand{\bk}{{\cal{K}}}
\newcommand{\by}{{\cal{Y}}}
\newcommand{\bx}{{\cal{X}}}
\newcommand{\bB}{{\cal{B}}}
\newcommand{\duof}{{\cal{F}}}
\newcommand{\oop}{\operatorname{op}}
\newcommand{\op}{^{\oop}}
\newcommand{\pps}{\operatorname{ps}}
\newcommand{\ps}{^{\pps}}
\newcommand{\lowps}{_{\pps}}
\newcommand{\ttract}{\operatorname{tract}}
\newcommand{\tract}{_{\ttract}}
\newcommand{\rrep}{\operatorname{rep}}
\newcommand{\rep}{_{\rrep}}

\newcommand{\ev}{\mathcal{V}}
\newcommand{\lom}{\ell{om}}
\newcommand{\rom}{r{om}}
\newcommand{\warmon}{\,\Box\,}
\numberwithin{equation}{section}

\begin{document}

\author{Thomas Booker\footnote{This author was supported by an Australian Postgraduate Award.}\, and Ross Street\footnote{This author gratefully acknowledges the support of an Australian Research Council Discovery Grant DP1094883.}}
\title{Tannaka duality and convolution for duoidal categories}
\date{}
\maketitle
\begin{center}
\end{center}
\begin{abstract}
\noindent Given a horizontal monoid $M$ in a duoidal category $\duof$, we examine the relationship between bimonoid structures on $M$ and monoidal structures on the category $\duof^{\ast M}$ of right $M$-modules which lift the vertical monoidal structure of $\duof$.
We obtain our result using a variant of the Tannaka adjunction.
The approach taken utilizes hom-enriched categories rather than categories on which a monoidal category acts (``actegories'').
The requirement of enrichment in $\duof$ itself demands the existence of some internal homs, leading to the consideration of convolution for duoidal categories.
Proving that certain hom-functors are monoidal, and so take monoids to monoids, unifies classical convolution in algebra and Day convolution for categories.
Hopf bimonoids are defined leading to a lifting of closed structures on $\duof$ to  $\duof^{\ast M}$.
Warped monoidal structures permit the construction of new duoidal categories.
\end{abstract}
\tableofcontents

\section{Introduction}

This paper initiates the development of a general theory of duoidal categories.
In addition to providing the requisite definition of a duoidal $\ev$-category, various ``classical'' concepts are reinterpreted and new notions put forth, including:  produoidal $\ev$-categories, convolution structures and duoidal cocompletion, enrichment in a duoidal $\ev$-category, Tannaka duality, lifting closed structures to a category of representations (Hopf opmonoidal monads), and discovering new duoidal categories by ``warping'' the monoidal structure of another.
Duoidal categories, some examples, and applications, have appeared in the Aguiar-Mahajan book [\ref{AgMaha}] (under the name ``2-monoidal categories''), in the recently published work of Batanin-Markl [\ref{BatMarkl}] and in a series of lectures by the second author [\ref{StreetUCLlectures}].
Taken together with this paper, the vast potential of duoidal category theory is only now becoming apparent.

An encapsulated definition is that a duoidal $\ev$-category $\duof$ is a pseudomonoid in the 2-category  Mon($\ev$-Cat)  of monoidal $\ev$-categories, monoidal $\ev$-functors and monoidal $\ev$-natural transformations.
Since Mon($\ev$-Cat) is equivalently the category of pseudomonoids in $\ev$-Cat we are motivated to call a pseudomonoid in a monoidal bicategory a \emph{monoidale} (i.e. a monoidal object).
Thus a duoidal $\ev$-category is an object of $\ev$-Cat equipped with two monoidal structures, one called horizontal and the other called vertical, such that one is monoidal with respect to the other.
Calling such an object a \emph{duoidale} encourages one to consider duoidales in other monoidal bicategories, in particular $\bm = \ev$-Mod.
By giving a canonical monoidal structure on the $\ev = \bm ({\cal I},{\cal I})$ valued-hom for any left unit closed monoidal bicategory $\bm$ (see Section \ref{monoidalityofhomsection}), we see that a duoidale in $\bm = \ev$-Mod is precisely the notion of promonoidal category lifted to the duoidal setting, that is, a produoidal $\ev$-category.

A study of duoidal cocompletion (in light of the produoidal $\ev$-category material) leads to Section \ref{enrichduoidsection} where we consider enrichment in a duoidal $\ev$-category base.
We observe that if $\duof$ is a duoidal $\ev$-category then the vertical monoidal structure $\circ$ lifts to give a monoidal structure on $\duof_{h}$-Cat.
If $\duof$ is then a horizontally left closed duoidal $\ev$-category then $\duof$ is in fact a monoidale $(\duof_{h} , \hat{\circ}, \ulcorner\mathbf{1}\urcorner )$ in $\duof_{h}$-Cat with multiplication $\hat{\circ} : \duof_{h}\circ\duof_{h}\longrightarrow\duof_{h}$ defined using the evaluation of homs.
That is, $\duof_{h}$ is an $\duof_{h}$-category.

Section \ref{revisitetannakaathome} revisits the Tannaka adjunction as it pertains to duoidal $\ev$-categories.
We write $\duof_{h}\textrm{-Cat}\downarrow\ps\duof_{h}$ for the $2$-category $\duof_{h}\textrm{-Cat}\downarrow\duof_{h}$ restricted to having $1$-cells those triangles that commute up to an isomorphism.
Post  composition with the monoidale multiplication $\hat{\circ}$ yields a tensor product $\overline{\circ}$ on $\duof_{h}\textrm{-Cat}\downarrow\ps\duof_{h}$ and we write $\duof\textrm{-Cat}\downarrow\ps\duof$ for this monoidal $2$-category.
Let $\duof^{\ast M}$ be the $\duof_{h}$-category of Eilenberg-Moore algebras for the monad $-\ast M$.
There is a monoidal functor $\textrm{mod} : (\textrm{Mon}\, \duof)\op \longrightarrow  \duof\textrm{-Cat}\downarrow\ps\duof$ defined by taking a monoid $M$ to the object $U_{M} : \duof^{\ast M}\longrightarrow\duof_{h}$.
Here Mon $\duof$ is only being considered as a monoidal category, not a $2$-category.
Representable objects of  $\duof\textrm{-Cat}\downarrow\ps\duof$ are closed under the monoidal structure $\overline{\circ}$ which motivates restricting to  $\duof\textrm{-Cat}\downarrow\ps\rep\duof$.
Since representable functors are ``tractable'' and the functor $\textrm{end} \, : \duof\textrm{-Cat}\downarrow\ps\rep\duof\longrightarrow \textrm{Mon }\duof$ is strong monoidal we have the biadjunction
$$\xymatrix{(\textrm{Bimon }\duof_{h})\op \ar@<-1.5ex>[rr]_-{\textrm{mod}}="1" && \ar@<-1.5ex>[ll]_-{\textrm{end}}="2" \ar@{}"1";"2"|-{\perp} \textrm{Mon}\lowps(\duof\textrm{-Cat}\downarrow\ps\rep\duof )}$$
giving the correspondence between bimonoid structures on $M$ and isomorphism classes of monoidal structures on $\duof^{\ast M}$ such that the underlying functor is strong monoidal into the vertical structure on $\duof$.
The non-duoidal version of this result is attributed to Bodo Pareigis (see [\ref{Bodo1}], [\ref{Bodo2}] and [\ref{Bodo3}]).

The notion of a Hopf opmonoidal monad is found in the paper of Brugui\`eres-Lack-Virelizier [\ref{BLV}].
We adapt their work to the duoidal setting in order to lift closed structures on the monoidale (monoidal $\duof_{h}$-category) $(\duof , \hat{\circ} ,  \ulcorner\mathbf{1}\urcorner )$ to the $\duof_{h}$-category of right modules $\duof^{\ast M}$ for a bimonoid $M$.
In particular, Proposition \ref{closureprop} says that a monoidal $\duof_{h}$-category  $(\duof , \hat{\circ} ,  \ulcorner\mathbf{1}\urcorner )$ is closed if and only if $\duof_{v}$ is a closed monoidal $\ev$-category and there exists $\ev$-natural isomorphisms $X\circ (W\ast Y) \cong W\ast (X\circ Y) \cong (W\ast X)\circ Y$.
In light of $\duof$ being a duoidal $\ev$-category, Proposition \ref{closurepropmodified} gives a refinement of this result which taken together with Proposition \ref{closureprop} yields two isomorphims
$$X\ast (J\circ Y) \cong X\circ Y\cong Y\ast (X\circ J)$$
and
$$Y\circ (W\ast \mathbf{1} ) \cong W\ast Y \cong (W\ast \mathbf{1})\circ Y\, .$$
This result implies that in order to know $\circ$ we only need to know $\ast$ and $J\circ -$ or $-\circ J$.
Similarly to know $\ast$ we need only know $\circ$ and $\mathbf{1}\ast - $ or $-\ast\mathbf{1}$.
This extreme form of interpolation motivates the material of Section \ref{warpedmonstructsect}.

We would like a way to generate new duoidal categories.
One possible method presented here is the notion of a \emph{warped monoidal structure}.
In its simplest presentation, a warping for a monoidal category $\ba = (\ba , \otimes )$ is a purtabation of $\ba$'s tensor product by a ``suitable'' endo-functor $T:\ba\longrightarrow\ba$ such that  the new tensor product is defined by
$$A\warmon B = TA\otimes B\, .$$
We lift this definition to the level of a monoidale $A$ in a monoidal bicategory $\bm$.
Proposition \ref{warpmonoidale} observes that a warping for a monoidale determines another monoidale structure on $A$.
If $\duof$ is a duoidal $\ev$-category satisfying the right-hand side of the second isomorphism above then a vertical warping of $\duof$ by $T = -\ast\mathbf{1}$ recovers $\duof_{h}$.
This is precisely a warping of the monoidale $\duof_{v}$ in $\bm = \ev$-Cat.
The last example given generates a duoidal category by warping the monoidal structure of any lax braided monoidal category viewed as a duoidal category with $\ast = \circ = \otimes$ and $\gamma = 1\otimes c \otimes 1$.

\newpage

\section{The monoidality of hom}\label{monoidalityofhomsection}

Let $(\ev ,\otimes)$ be a symmetric closed complete and cocomplete monoidal category.
Recall from [\ref{KellyBook}] that a $\ev$-natural transformation $\theta$ between $\ev$-functors $T,S: \ba\longrightarrow\bx$ consists of a \emph{$\ev$-natural family} 
$$\xymatrix{\theta_{A} : TA \ar[r] & SA,\, A\in\ba} ,$$
such that the diagram
$$\xymatrix{
\ba (A, B) \ar[rr]^-{T} \ar[d]_-{S} && \bx ( TA, TB) \ar[d]^-{\bx (1, \theta_{B})} \\
\bx (SA, SB) \ar[rr]_-{\bx (\theta_{A}, 1)} && \bx (TA, SB)
}$$
commutes in the base category $\ev$.

If $(\bc , \boxtimes)$ is a monoidal $\ev$-category with tensor product $\boxtimes $ then the associativity isomorphisms $a_{A.B,C} : (A \boxtimes B) \boxtimes C \longrightarrow A \boxtimes (B \boxtimes C)$ are necessarily a $\ev$-natural family, which amounts to the commutativity of the diagram
$$\xymatrix{
(\bc (A,A')\otimes \bc (B,B'))\otimes \bc (C,C') \ar[d]_-{\cong} \ar[rr]^-{\boxtimes(\boxtimes\otimes 1)} \ar@{}[rrdd]|-{Nat_{a}} && \bc ((A \boxtimes B) \boxtimes C , (A'\boxtimes B' ) \boxtimes C') \ar[dd]^-{\bc(1,a_{A',B',C'})} \\
\bc (A,A')\otimes (\bc (B,B')\otimes \bc (C,C')) \ar[d]_-{\boxtimes(1\otimes \boxtimes)} && \\
\bc (A \boxtimes (B \boxtimes C) , A'\boxtimes (B' \boxtimes C')) \ar[rr]_-{\bc (a_{A,B,C}, 1)} && \bc ((A \boxtimes B) \boxtimes C , A'\boxtimes (B' \boxtimes C'))
}$$
Similarly the $\ev$-naturality of the unit isomorphisms
$$\xymatrix{\ell_{A} :I\boxtimes A\ar[r] & A\qquad \textrm{and}\qquad r_{A} : A\boxtimes I\ar[r] & A}$$
amounts to the commutativity of
$$\xymatrix{
\bc (A,A') \ar@{=}[d] \ar[r]^-{I\boxtimes -} \ar@{}[rd]|-{Nat_{\ell}} & \bc (I\boxtimes A, I\boxtimes A') \ar[d]^-{\bc(1, \ell_{A'})} && \bc (A,A') \ar[r]^-{-\boxtimes I}\ar@{=}[d] \ar@{}[rd]|-{Nat_{r}} & \bc(A\boxtimes I, A'\boxtimes I) \ar[d]^-{\bc(1, r_{A'})} \\
\bc(A, A') \ar[r]_-{\bc(\ell_{A} , 1)} & \bc (I\boxtimes A, A') && \bc (A, A') \ar[r]_-{\bc(r_{A} , 1)} & \bc ( A\boxtimes I, A')
}$$

\begin{proposition}\label{monoidalhomprop}
If $(\bc,\boxtimes )$ is a monoidal $\ev$-category then the $\ev$-functor 
$$\xymatrix{\bc (- , -): \bc\op\otimes\bc \ar[r] & \ev}$$
is equipped with a canonical monoidal structure.
\end{proposition}
\begin{proof}
For $\bc (-,-)$ to be monoidal we require the morphisms
$$\xymatrix{\boxtimes : \bc (W, X) \otimes \bc (Y, Z) \ar[r] & \bc (W\boxtimes Y , X \boxtimes Z)}$$
and
$$\xymatrix{j_{I} : I \ar[r] & \bc( I,I)}$$
to satisfy the axioms
$$\xymatrix{
(\bc (U,V)\otimes \bc (W,X))\otimes \bc (Y,Z) \ar[rr]^-{\boxtimes \otimes 1} \ar[d]_-{\cong} && \bc (U\boxtimes W , V\boxtimes X ) \otimes \bc (Y,Z) \ar[d]^-{\boxtimes} \\
\bc(U,V) \otimes  (\bc (W,X)\otimes \bc (Y, Z) ) \ar[d]_-{1\otimes\boxtimes} && \bc ((U\boxtimes W) \boxtimes Y, (V\boxtimes X) \boxtimes Z) \ar[d]^-{\bc (a_{U,W,Y}^{-1}, a_{V,X,Z})} \\
\bc (U,V) \otimes \bc (W\boxtimes Y, X\boxtimes Z) \ar[rr]_-{\boxtimes} && \bc (U\boxtimes (W\boxtimes Y) , V\boxtimes (X\boxtimes Z))
}$$
and
$$\scalebox{0.9}{\xymatrix{
\bc (I, I) \otimes \bc (Y, Z) \ar[r]^-{\boxtimes} & \bc (I \boxtimes Y , I \boxtimes Z) \ar[d]^-{\bc (\ell_{Y}^{-1},\ell_{Z} )}  & \bc (W,X) \otimes I \ar[r]^-{r} \ar[d]_-{1\otimes j_{I}} & \bc (W,X) \\
I\otimes\bc (Y, Z) \ar[r]_-{\ell} \ar[u]^-{j_{I} \otimes 1} & \bc (Y, Z) & \bc (W,X) \otimes \bc (I,I) \ar[r]_-{\boxtimes} & \bc (W \boxtimes I, X \boxtimes I)  \ar[u]_-{\bc (r_{W}^{-1},r_{X} )}
}}$$
These diagrams are simply reorganizations of the diagrams $Nat_{a}$, $Nat_{\ell}$, and $Nat_{r}$ above. \qed
\end{proof}

\begin{corollary}
If $C$ is a comonoid and $A$ is a monoid in the monoidal $\ev$-category $\bc$ then $\bc (C ,A)$ is canonically a monoid in $\ev$.
\end{corollary}
\begin{proof}
We observe that monoidal $\ev$-functors take monoids to monoids and $(C,A)$ is a monoid in $\bc\op\otimes\bc$.\qed
\end{proof}

\begin{proposition}
If $\bc$ is a braided monoidal $\ev$-category then
$$\xymatrix{\bc (-,-) : \bc\op\otimes\bc \ar[r] & \ev}$$
is a braided monoidal $\ev$-functor.
\end{proposition}
\begin{proof}
Let $c_{X,Y} : X\boxtimes Y \longrightarrow Y\boxtimes X$ denote the braiding on $\bc$.
The requirement of $\ev$-naturality for this family of isomorphisms amounts precisely to the commutativity of
$$\xymatrix{
\bc (W,X) \otimes \bc (Y,Z) \ar[r]^-{\boxtimes} \ar[d]_-{\cong} & \bc (W\boxtimes Y, X\boxtimes Z) \ar[d]^-{\bc (c^{-1},c)} \\
\bc (Y,Z) \otimes \bc (W,X) \ar[r]_-{\boxtimes} & \bc (Y\boxtimes W , Z\boxtimes X)
}$$
which is exactly the braiding condition for the monoidal functor $\bc (-,-)$ of Proposition \ref{monoidalhomprop}. \qed
\end{proof}

We now give a spiritual successor to the above by moving to the level of monoidal bicategories.

\begin{proposition}\label{canonicalmonoidal}
If $\bm$ is a monoidal bicategory then the pseudofunctor 
$$\xymatrix{\bm (-,-) : \bm\op\times\bm \ar[r] & \mathrm{Cat}}$$
is equipped with a canonical monoidal structure.
\end{proposition}
\begin{proof}
We avail ourselves of the coherence theorem of [\ref{GPS}] by assuming that $\bm$ is a Gray monoid (see [\ref{DayStreet}]).
The definition of a monoidal pseudofunctor (called a ``weak monoidal homomorphism'') between Gray monoids is defined on pages 102 and 104 of [\ref{DayStreet}].
Admittedly Cat is not a Gray monoid, but the adjustment to compensate for this is not too challenging.

In the notation of [\ref{DayStreet}], the pseudonatural transformation $\chi$ is defined at objects to be the functor
$$\xymatrix{\otimes : \bm (A,A') \times \bm (B,B') \ar[r] & \bm (A\otimes B, A'\otimes B')}$$
and at the morphisms to be the isomorphism
$$\xymatrix{
\bm (A,A') \times \bm (B,B') \ar@{}[rrd]|-{\cong} \ar[rr]^-{\otimes} \ar[d]_-{\bm (f,f')\times\bm (g,g')} && \bm (A\otimes B, A'\otimes B') \ar[d]^-{\bm (f,g)\times\bm (f',g')} \\
\bm (C,C') \times \bm (D,D') \ar[rr]_-{\otimes} && \bm (C\otimes D, C'\otimes D')  
}$$
whose component
$$(f'uf)\otimes (g'vg) \cong (f'\otimes g')(u\otimes v)(f\otimes g)$$
at $(u,v)\in\bm (A,A')\times\bm (B,B')$ is the canonical isomorphism associated with the pseudofunctor $\otimes : \bm\times\bm \longrightarrow\bm$ (see the top of page 102 of [\ref{DayStreet}]).
For $\iota$, we have the functor $1 \longrightarrow\bm (I,I)$ which picks out $1_{I}$.
For $\omega$, we have the natural isomorphism
$$\xymatrix{
\bm (A, A') \times \bm (B, B') \times \bm (C,C') \ar[rr]^-{\otimes\times 1}_-*!/d0.5pt/{\phantom{A}}="1" \ar[d]_-{1\times\otimes} && \bm (A\otimes B, A'\otimes B') \times \bm (C,C') \ar[d]^-{\otimes} \\
\bm (A,A') \times \bm (B\otimes C, B'\otimes C') \ar[rr]_-{\otimes}^-*!/u0.5pt/{\phantom{A}}="2" \ar@{=>}"1";"2"_*!/l3pt/{\labelstyle{\omega}} && \bm (A\otimes B\otimes C, A' \otimes B'\otimes C')
}$$
whose component at $(u,v,w)$ is the canonical isomorphism
$$(u\otimes v)\otimes w \cong u\otimes (v\otimes w)$$
associated with $\otimes :\bm\times\bm \longrightarrow\bm$.
For $\xi$ and $\kappa$, we have the natural isomorphisms
$$\xymatrix{
\ar@{}[rr]_-*!/d0.5pt/{\phantom{A}}="1" & \bm( A,A')\times \bm (I,I) \ar[rd]^-*!/r3pt/{\labelstyle{\otimes}} & \\
\bm (A,A') \ar[ru]^-*!/l3pt/{\labelstyle{1\times \ulcorner 1_{I} \urcorner} } \ar[rr]_-{1}^-*!/d0.5pt/{\phantom{A}}="2" \ar@{=>}"1";"2"_-*!/l3pt/{\labelstyle{\cong}}  && \bm (A,A') 
}$$
and
$$\xymatrix{
\ar@{}[rr]_-*!/d0.5pt/{\phantom{A}}="1" & \bm( I,I)\times \bm (A,A') \ar[rd]^-*!/r3pt/{\labelstyle{\otimes}}  & \\
 \bm (A,A')  \ar[ru]^-*!/l3pt/{\labelstyle{\ulcorner 1_{I} \urcorner \times1}} \ar[rr]_-{1}^-*!/d0.5pt/{\phantom{A}}="2" \ar@{=>}"1";"2"_-*!/l3pt/{\labelstyle{\cong}} && \bm (A,A')
}$$
with canonical components
$$u\otimes 1_{I} \cong u\qquad \textrm{and} \qquad 1_{I}\otimes u\cong u\, .$$
The two required axioms are then a consequence of the coherence conditions for pseudofunctors in the case of $\otimes : \bm\times\bm \longrightarrow\bm$. \qed
\end{proof}

\begin{corollary}
([\ref{DayStreet}]; page 110, Proposition 4) If $A$ is a pseudomonoid and $C$ is a pseudocomonoid in a monoidal bicategory $\bm$ then the category $\bm (C, A)$ is equipped with a canonical monoidal structure.
\end{corollary}

\begin{proposition}\label{braidedmonpseufunc}
If $\bm$ is a braided monoidal bicategory then 
$$\xymatrix{\bm (-,-) : \bm\op\times\bm \ar[r] & Cat}$$
is a braided monoidal pseudofunctor.
\end{proposition}
\begin{proof}
The required data of page 122, Definition 14 in [\ref{DayStreet}] is provided by the invertible modification
$$\xymatrix{
\bm (A,A')\times \bm ( B, B') \ar[rr]^-{\cong} \ar[d]_-{\otimes}^-{}="1" && \bm (B,B')\times \bm (A,A') \ar[d]^-{\otimes}_-{}="2" \ar@{}"1";"2"|-{\Longrightarrow}^-{\cong} \\
\bm (A\otimes B, A'\otimes B') \ar[rr]_-{\bm (\rho^{-1},\rho)} && \bm (B\otimes A, B'\otimes A')
}$$
whose component at $(u,v)$ is
$$\xymatrix{
B\otimes A  \ar[rr]^-{\rho}_-{\phantom{.}}="a" \ar@/^-3.5ex/[drr]_-{1} && A\otimes B \ar[d]_-{\rho}  \ar[rr]^-{u\otimes v}_-{\phantom{.}}="1"  && A'\otimes B' \ar[d]^-{\rho} \\
\ar@{}[rr]^-{\phantom{.}}="b" \ar@{}"a";"b"^<<<<<*!/r5pt/{\labelstyle{\cong}} && B\otimes A \ar[rr]_-{v\otimes u}^-{\phantom{.}}="2" \ar@{}"1";"2"|-{\cong}^-*!/r3pt/{\labelstyle{\rho_{u,v}}}  && B'\otimes A'
}$$ \qed
\end{proof}

What we really want is a presentation of these results lifted to the level of enriched monoidal bicategories.

Suppose $\bm$ is a monoidal bicategory.
Put $\ev = \bm (I, I)$, regarding it as a monoidal category under composition $\circ$.
There is another ``multiplication'' on $\ev$ defined by the composite 
$$\xymatrix{\bm (I,I) \times \bm (I,I) \ar[r]^-{\otimes} & \bm (I\otimes I,I\otimes I) \ar@{}[r]|-{\cong} & \bm (I,I)}$$
with the same unit $1_{I}$ as $\circ$.
By Proposition 5.3 of [\ref{JoyalStreet}], a braiding is obtained on $\ev$.

Furthermore, each hom category $\bm (X,Y)$ has an action 
$$\xymatrix{\bm (I,I) \times \bm (X,Y) \ar[r]^-{\otimes} & \bm (I\otimes X,I\otimes Y) \ar@{}[r]|-{\simeq} & \bm (X,Y)}$$ 
by $\ev$ which we abusively write as 
$$\xymatrix{(v , m) \ar@{|->}[r] & v \otimes m} .$$
We call $\bm$ \emph{left unit closed} when each functor 
$$\xymatrix{- \otimes m : \ev \ar[r] & \bm (X,Y)}$$
has a right adjoint 
$$\xymatrix{[m,-] : \bm (X,Y) \ar[r] & \ev} .$$
That is, we have a natural isomorphism
$$\bm (X,Y) (v\otimes m,n) \cong \ev (v,[m,n])\,.$$
In particular, this implies $\ev$ is a left closed monoidal category and that each hom category $\bm (X,Y)$ is $\ev$-enriched with $\ev$-valued hom defined by $[m,n]$.
Furthermore, since $\ev$ is braided, the $2$-category $\ev$-Cat of $\ev$-categories, $\ev$-functors and $\ev$-natural transformations is monoidal; see Remark 5.2 of [\ref{JoyalStreet}].

\begin{proposition}\label{liftingtovcat}
If the monoidal bicategory $\bm$ is left unit closed then the monoidal pseudofunctor of Proposition \ref{canonicalmonoidal} lifts to a monoidal pseudofunctor 
$$\xymatrix{\bm (-,-) : \bm\op\times\bm \ar[r] & \ev \textrm{-} \mathrm{Cat}}$$
where $\ev = \bm (I,I)$ as above.
\end{proposition}
\begin{proof}
We use the fact that, for tensored $\ev$-categories $\ba$ and $\bB$, enrichment of a functor $F:\ba \longrightarrow\bB$ to a $\ev$-functor can be expressed in terms of a lax action morphism structure
$$\xymatrix{\overline{\chi}_{V,A}:V\otimes FA \ar[r] & F(V\otimes A)}$$
for $V\in\ev$, $A\in\ba$.
Given such $\ev$-functors $F,G:\ba \longrightarrow\bB$, a family of morphisms 
$$\xymatrix{\theta_{A} : FA \ar[r] & GA}$$
is $\ev$-natural if and only if the diagrams 
$$\xymatrix{
V\otimes FA \ar[rr]^-{\overline{\chi}_{V,A}} \ar[d]_-{1\otimes\theta_{A}} && F(V\otimes A) \ar[d]^-{\theta_{V\otimes A}} \\
V\otimes GA \ar[rr]_-{\overline{\chi}_{V,A}} && G(V\otimes A)
}$$
commute.
Therefore, to see that the functors 
$$\xymatrix{\bm (f,g) : \bm (X,Y) \ar[r] & \bm (X',Y')} ,$$
for $f:X'\longrightarrow X$ and $g:Y \longrightarrow Y'$, are $\ev$-enriched, we require $2$-cells 
$$\xymatrix{v\otimes (g\circ m\circ f) \ar[r] & g\circ (v\otimes m) \circ f }$$
which constiture a lax action morphism.
As in the proof of Proposition \ref{canonicalmonoidal}, we assume that $\bm$ is a Gray monoid where we can take these $2$-cells to be the canonical isomorphisms.
It is then immediate that the $2$-cells $\sigma : f \Longrightarrow f'$ and $\tau : g \Longrightarrow g'$ induce $\ev$-natural transformations $\bm (\sigma , \tau ): \bm (f,g) \Longrightarrow \bm (f',g')$.

For the monoidal structure on $\bm (-,-)$, we need to see that the efect of the tensor of $\bm$ on homs defines a $\ev$-functor 
$$\xymatrix{\otimes : \bm (A,A') \otimes \bm (B,B') \ar[r] & \bm (A\otimes B, A'\otimes B')} .$$
Again we make use of the coherent isomorphisms; in this case they are
$$v\otimes (m\otimes n) \cong (v\circ m)\otimes (v\circ n)$$
for $v: I \longrightarrow I$, $m:A \longrightarrow A'$, $n: B \longrightarrow B'$.
It is clear that $\iota$ can be regarded as a $\ev$-functor $\iota : {\cal I} \longrightarrow \bm (I,I)$.
The $\ev$-naturality of all the $2$-cells involved in the monoidal structure on $\bm (-,-)$ now follows automatically from the naturality of the Gray monoid constraints. \qed
\end{proof}

\begin{proposition}\label{symmetriclift}
In the situation of Proposition \ref{liftingtovcat}, if $\bm$ is also symmetric then so is $\bm (-,-)$.
\end{proposition}
\begin{proof}
If $\bm$ is symmetric, so too is $\ev = \bm (I,I)$.
Consequently, $\ev$-Cat is also symmetric.
Referring to the proof of Proposition \ref{braidedmonpseufunc}, we see that the techniques of the proof of Proposition \ref{liftingtovcat} apply. \qed
\end{proof}

\begin{example}
Let $\ev$ be any braided monoidal category which is closed complete and cocomplete.
Put $\bm = \ev\textrm{-Mod}$, the bicategory of $\ev$-categories, $\ev$-modules (i.e. $\ev$-distributors or equivalently $\ev$-profunctors), and $\ev$-module morphisms.
This $\bm$ is a well-known example of a monoidal bicategory \ref{DayStreet}.
We can easily identify $\ev$ with $\ev$-Mod$({\cal I},{\cal I})$ and the action on $\bm (\ba , \bx)$ with the functor
$$\xymatrix{\ev\times\ev\textrm{-Mod}(\ba ,\bx ) \ar[r] & \ev\textrm{-Mod} (\ba ,\bx )}$$
given by the mapping
$$\xymatrix{(V,M) \ar@{|->}[r] & V\otimes M}$$
defined by $(V\otimes M )(X,A) = V\otimes M(X,A)$ with left module action
$$\xymatrix{\ba (A,B)\otimes V\otimes M(X,A) \ar[r]^-{c\otimes 1}_-{\cong} & V\otimes \ba (A,B) \otimes M(X,A) \ar[r]^-{1\otimes act_{\ell}} & V\otimes M(X,B)}$$
and right module action
$$\xymatrix{V\otimes M(X,A)\otimes \bx (Y,X) \ar[r]^-{1\otimes_{r}} & V\otimes M(Y,A)\, ,}$$
where $c$ is the braiding of $\ev$ and we have ignored associativity isomorphisms.
To see that $\bm = \ev\textrm{-Mod}$ is left unit closed we easily identify $[M,N]\in\ev$ for $M,N\in\ev\textrm{-Mod}(\ba ,\bx )$ with the usual $\ev$-valued hom for the $\ev$-category $[\bx\op\otimes\ba ,\ev]$; namely,
$$[M,N] = \int_{X,A} [M(X,A),N(X,A)]\, ,$$
the ``object of $\ev$-natural transformations''.
Therefore, in this case, Proposition \ref{liftingtovcat} is about the pseudofunctor
$$\xymatrix{\ev\textrm{-Mod}\op\times\ev\textrm{-Mod} \ar[r] & \ev\textrm{-Cat}} ,$$
given by the mapping
$$\xymatrix{(\ba , \bx ) \ar@{|->}[r] & [\bx\op\otimes\ba ,\ev ]} ,$$
asserting monoidality.
When $\ev$ is symmetric, Proposition \ref{symmetriclift} assures us the pseudofunctor is also symmetric.
\end{example}
\begin{remark}
There is presumably a more general setting encompassing the results of this section.
For a monoidal bicategory $\bk$, it is possible to define a notion of \emph{$\bk$-bicategory} $\bm$ by which we mean that the homs $\bm (X,Y)$ are objects of $\bk$.
For Proposition \ref{monoidalhomprop} we would take $\bk$ to be $\ev$ as a locally discrete bicategory and $\bm$ to be $\bc$.
For Proposition \ref{canonicalmonoidal}, $\bk$ would be Cat.
For Proposition \ref{liftingtovcat}, $\bk$ would be $\ev$-Cat.
Then, as in these cases, we would require $\bk$ to be braided in order to define the \emph{tensor product of $\bk$-bicategories} and so \emph{monoidal $\bk$-bicategories}.
With all this properly defined, we expect
$$\xymatrix{\bm (-,-) :\bm\op\otimes\bm \ar[r] & \bk}$$
to be a monoidal $\bk$-pseudofunctor.
\end{remark}

\section{Duoidal $\ev$-categories}

Throughout $\ev$ is a symmetric monoidal closed, complete and cocomplete category.
The following definition agrees with that of Batanin and Markl in [\ref{BatMarkl}] and, under the name $2$-monoidal category, Aguiar and Mahajan in [\ref{AgMaha}].

\begin{definition}\label{duocatdef}
A \emph{duoidal structure} on a $\ev$-category $\duof$ consists of two $\ev$-monoidal structures
\begin{eqnarray}
\xymatrix{\ast : \duof\otimes\duof \ar[r] & \duof , \quad \ulcorner J\urcorner:\mathbf{1} \ar[r] & \duof , \label{starfunc}} \\
\xymatrix{\circ :\duof\otimes\duof \ar[r] & \duof , \quad \ulcorner 1\urcorner:\mathbf{1} \ar[r] & \duof , \label{circfunc}}
\end{eqnarray}
such that either of the following equivalent conditions holds:
\begin{enumerate}
\item[(i)] the $\ev$-functors $\circ$ and $\ulcorner 1\urcorner$ of (\ref{circfunc}) and their coherence isomorphisms are monoidal with respect to the monoidal $\ev$-category $\duof_{h}$ of (\ref{starfunc}). \label{circview}
\item[(ii)] the $\ev$-functors $\ast$ and $\ulcorner J\urcorner$ of (\ref{starfunc}) and their coherence isomorphisms are opmonoidal with respect to the monoidal $\ev$-category $\duof_{v}$ of (\ref{circfunc}). \label{starview}
\end{enumerate}
We call the monoidal $\ev$-category $\duof_{h}$ of (\ref{starfunc}) \emph{horizontal} and the monoidal $\ev$-category $\duof_{v}$ of (\ref{circfunc}) \emph{vertical}; this terminology comes from an example of derivation schemes due to [\ref{BatMarkl}] (also see [\ref{StreetUCLlectures}]).

\noindent The extra elements of structure involved in (i) and (ii) are a $\ev$-natural middle-of-four interchange transformation
$$\xymatrix{\gamma: (A\circ B) \ast (C \circ D) \ar[r] & (A\ast C) \circ (B\ast D)},$$
and maps
$$\xymatrix{\mathbf{1}\ast\mathbf{1} \ar[r]^-{\mu} & \mathbf{1} & \ar[l]_-{\tau} J \ar[r]^-{\delta} & J \circ J}$$
such that the diagrams
\begin{eqnarray}\label{gammaaxiom1}
\scalebox{0.9}{\xymatrix{ 
((A\circ B) \ast (C \circ D)) \ast (E\circ F) \ar[d]_-{\gamma \ast 1} \ar[rr]^-{\cong} && (A\circ B) \ast ((C \circ D) \ast (E\circ F)) \ar[d]^-{1\ast\gamma}\\
((A\ast C) \circ (B \ast D)) \ast (E\circ F) \ar[d]_-{\gamma } && (A\circ B) \ast ((C \ast E) \circ (D\ast F)) \ar[d]^-{\gamma} \\
((A \ast C) \ast E ) \circ ((B \ast D)\ast F) \ar[rr]_-{\cong}&& (A \ast (C \ast E )) \circ (B \ast (D\ast F))
}}
\end{eqnarray}
\begin{eqnarray}\label{gammaaxiom2}
\scalebox{0.9}{\xymatrix{ 
((A\circ B) \circ C ) \ast ((D\circ E ) \circ F) \ar[rr]^-{\cong} \ar[d]_-{\gamma} && (A\circ (B \circ C) ) \ast (D\circ (E  \circ F)) \ar[d]^-{\gamma} \\
((A\circ B ) \ast (D \circ E ) ) \circ (C \ast F) \ar[d]_-{\gamma\circ 1} && ( A \ast D) \circ  ((B \circ C) \ast (E\circ F)) \ar[d]^-{1\circ \gamma} \\
((A\ast D ) \circ (B \ast E ) ) \circ (C \ast F) \ar[rr]_-{\cong} && ( A \ast D) \circ  ((B \ast E) \circ (C\ast F)) 
}}
\end{eqnarray}
and
\begin{eqnarray}\label{deltaaxiom}
\scalebox{0.9}{\xymatrix{ 
J\ast (A\circ B) \ar[r]^-{\delta \ast 1} & (J\circ J)\ast (A\circ B) \ar[d]^-{\gamma}  & (A\circ B) \ast J \ar[r]^-{1\ast\delta} & (A\circ B) \ast (J\circ J) \ar[d]^-{\gamma} \\
A\circ B \ar[u]^-{\cong} \ar[r]_-{\cong} & (J \ast A) \circ (J \ast B) & A\circ B \ar[u]^-{\cong} \ar[r]_-{\cong} & (A\ast J) \circ (B\ast J)
}}
\end{eqnarray}
\begin{eqnarray}\label{muaxiom}
\scalebox{0.9}{\xymatrix{ 
\mathbf{1}\circ (A\ast B)  & \ar[l]_-{\mu \circ 1} (\mathbf{1}\ast \mathbf{1})\circ (A\ast B)  & (A\ast B) \circ \mathbf{1}   & \ar[l]_-{1\circ \mu} (A\ast B)\circ (\mathbf{1} \ast \mathbf{1}) \\
A\ast B \ar[u]^-{\cong}  \ar[r]_-{\cong} & (\mathbf{1} \circ A) \ast (\mathbf{1}\circ B) \ar[u]^-{\gamma} & A\ast B  \ar[u]^-{\cong}  \ar[r]_-{\cong} & (A\circ \mathbf{1}) \ast (B\circ \mathbf{1}) \ar[u]^-{\gamma}
}}
\end{eqnarray}
commute, together with the requirement that $(\mathbf{1},\mu , \tau )$ is a monoid in $\duof_{h}$ and $(J,\delta , \tau )$ is a comonoid in $\duof_{v}$.
\end{definition}

\begin{example}
A braided monoidal category $\bc$ with braid isomorphism $c : A \otimes B \cong B\otimes A$ is an example of a duoidal category with $\otimes = \ast = \circ$ and $\gamma$, determined by $1_{A}\otimes c\otimes 1_{D}$ and re-bracketing, invertible.
\end{example}

\begin{example}
Let $\bc$ be a monoidal $\ev$-category. 
An important example is the $\ev$-category $\duof = [\bc\op\otimes\bc , \ev]$ of $\ev$-modules $\xymatrix{\bc \ar[r]|*=0-{\textrm{\scriptsize{|}}} & \bc}$ and $\ev$-module homomorphisms.
We see that $\duof$ becomes a duoidal $\ev$-category with $\ast$ the convolution tensor product for $\bc\op\otimes\bc$ and $\circ$ the tensor product ``over $\bc$''.
This example can be found in [\ref{StreetUCLlectures}].
\end{example}

\begin{definition}
A \emph{duoidal functor} $F: \duof \longrightarrow\duof '$ is a functor  $F$ that  is equipped with monoidal structures $\duof_{h} \longrightarrow\duof_{h}'$ and $\duof_{v} \longrightarrow\duof_{v}'$ which are compatible with the duoidal data $\gamma$, $\mu$, $\delta$, and $\tau$.
\end{definition}

\begin{definition}
A \emph{bimonoidal functor} $T: \duof \longrightarrow\duof '$ is a functor $F$ that is equipped with a monoidal structure $\duof_{h} \longrightarrow\duof_{h}'$ and an opmonoidal structure $\duof_{v} \longrightarrow\duof_{v}'$ both of which are compatible with the duoidal data $\gamma$, $\mu$, $\delta$, and $\tau$.
\end{definition}

\begin{definition}\label{bimonoidinf}
A \emph{bimonoid} $A$ in a duoidal category $\duof$ is a bimonoidal functor $\ulcorner A \urcorner : \mathbf{1} \longrightarrow \duof$.
That is, it is an object $A$ equipped with the structure of a monoid for $\ast$ and a comonoid for $\circ$, compatible via the axioms
\begin{eqnarray}
\xymatrix{A\ast A \ar[r]^-{\mu} \ar[d]_-{\delta \ast \delta} & A \ar[r]^-{\delta} \ar@{}[d]|-{=} & A\circ A \\ (A\circ A)\ast (A\circ A) \ar[rr]_-{\gamma} && (A\ast A)\circ (A\ast A) \ar[u]_-{\mu\circ\mu}}
\end{eqnarray}
\begin{eqnarray}
\xymatrix{A\ast A \ar[rr]^-{\mu} \ar[d]_-{\epsilon\ast\epsilon}^-{}="1" && A \ar[d]^-{\epsilon}_-{}="2" \ar@{}"1";"2"|-{=} \\  \mathbf{1}\ast \mathbf{1} \ar[rr]_-{\mu} && \mathbf{1}} & \xymatrix{ J\circ J \ar[d]_-{\eta\circ\eta}^-{}="1" && \ar[ll]_-{\delta} J \ar[d]^-{\eta}_-{}="2" \ar@{}"1";"2"|-{=} \\  A\circ A && \ar[ll]^-{\delta}  A}
\end{eqnarray}
\begin{eqnarray}
\xymatrix{J \ar[dd]_-{\tau} \ar[rd]^-{\eta} & \\ &\ar@{}[l]|-{=} A \ar[ld]^-{\epsilon} \\ \mathbf{1} &.} 
\end{eqnarray}
These are a lifting of the usual axioms for a bimonoid in a braided monoidal category.
\end{definition}

\section{Duoidales and produoidal $\ev$-categories}

Recall the two following definitions and immediately following example from [\ref{DayStreet}] where $\bm$ is a monoidal bicategory.

\begin{definition}
A \emph{pseudomonoid} $A$ in $\bm$ is an object $A$ of $\bm$ together with multiplication and unit morphisms $\mu : A\otimes A \longrightarrow A$, $\eta : I \longrightarrow A$, and invertible $2$-cells $a : \mu (\mu\otimes 1) \Longrightarrow \mu (1\otimes\mu)$, $\ell : \mu (\eta\otimes 1) \Longrightarrow 1$, and $r : \mu (1\otimes\eta)\Longrightarrow 1$ satisfying the coherence conditions given in [\ref{DayStreet}].
\end{definition}

\begin{definition}
A (lax-)morphism $f$ between pseudomonoids $A$ and $B$ in $\bm$ is a morphism $f: A \longrightarrow B$ equipped with
$$\xymatrix{M\otimes M \ar[r]^-{\mu} \ar[d]_-{f\otimes f}^-{}="1" & M \ar[d]^-{f}_-{}="2" \\ N\otimes N \ar[r]_-{\mu} \ar@{}"1";"2"|-{\Longrightarrow}^-{\varphi} & N}$$
and
$$\xymatrix{I \ar@/^2ex/[rrd]^-{\eta} \ar@/_5ex/[rrdd]_-{\eta} && \\  \ar@{}[rr]|{\Longrightarrow}^{\varphi_{0}} && M \ar[d]^-{f} \\  && N}$$
subject to three axioms.
\end{definition}

\begin{example}
If $\bm$ is the cartesian closed $2$-category of categories, functors, and natural transformations then a monoidal category is precisely a pseudomonoid in $\bm$.
\end{example}

This example motivates calling a pseudomonoid in a monoidal bicategory $\bm$ a \emph{monoidale} (short for a monoidal object of $\bm$).
A morphism $f: M\to N$ of monoidales is then a morphism of pseudomonoids (i.e. a monoidal morphism between monoidal objects).
We write Mon($\bm$) for the $2$-category of monoidales in $\bm$, monoidal morphisms, and monoidal $2$-cells. 
If $\bm$ is symmetric monoidal then so is  Mon($\bm$).

\begin{definition}
A \emph{duoidale} $F$ in $\bm$ is an object $F$ together with two monoidale structures
\begin{eqnarray}
\xymatrix{\ast : F \otimes F \ar[r] & F,\quad J : I \ar[r] & F}\\
\xymatrix{\circ : F \otimes F \ar[r] & F,\quad \mathbf{1} : I \ar[r] & F}
\end{eqnarray}
such that $\circ$ and $\mathbf{1}$ are monoidal morphisms with respect to $\ast$ and $J$.
\end{definition}

\begin{remark}
If $\bm$ = $\ev$-Cat then a duoidale in $\bm$ is precisely a duoidal $\ev$-category.
\end{remark}

Let $\bm$ = $\ev$-Mod be the symmetric monoidal bicategory of $\ev$-categories, $\ev$-modules, and $\ev$-module morphisms.
By Proposition \ref{symmetriclift}, there is a symmetric monoidal pseudofunctor
$$\xymatrix{\bm (\bid , -) : \bm \ar[r] & \ev\textrm{-Cat}}$$
defined by taking a $\ev$-category $\ba$ to the $\ev$-category $[\ba\op , \ev]$ of $\ev$-functors and $\ev$-natural transformations.

\begin{definition}
A \emph{produoidal $\ev$-category} is a duoidale in $\ev$-Mod.
\end{definition}

If $\duof$ is a produoidal $\ev$-category then there are $\ev$-modules 
\begin{eqnarray*}
\xymatrix{S : \duof \otimes \duof \ar[r] \ar@{}[r]|-{\scriptsize{|}} & \duof, \qquad H : {\cal I} \ar[r] \ar@{}[r]|-{\scriptsize{|}} & \duof} , \\
\xymatrix{R : \duof \otimes \duof \ar[r] \ar@{}[r]|-{\scriptsize{|}} & \duof, \qquad K : {\cal I}  \ar[r] \ar@{}[r]|-{\scriptsize{|}}& \duof} ,
\end{eqnarray*}
where $R$ and $K$ are monoidal with respect to $S$ so that there are $2$-cells $\gamma$, $\delta$, and $\tau$:
$$\xymatrix{
\duof\otimes\duof\otimes\duof\otimes\duof \ar@{}[r]|-{\simeq} \ar@{}[r]^-{1\otimes c \otimes 1} \ar[d] \ar@{}[d]|-{\rotatebox{90}{\scriptsize{|}}}_-{R\otimes R\,}^-{}="1" & \duof\otimes\duof\otimes\duof\otimes\duof \ar[r] \ar@{}[r]|-{\scriptsize{|}}^-*!/u1.5pt/{\labelstyle{S\otimes S}} &\duof\otimes\duof  \ar[d] \ar@{}[d]|-{\rotatebox{90}{\scriptsize{|}}}^-{\, R}_-{}="2" \ar@{}"1";"2"|-{\Longrightarrow}^-{\gamma} \\
\duof\otimes\duof  \ar[rr] \ar@{}[rr]|-{\scriptsize{|}}_-*!/d1.5pt/{\labelstyle{S}} && \duof
}$$
$$\xymatrix{
{\cal I} \ar@{}[r]|-{\simeq} \ar@/^-2ex/[rrd]|*=0-{\rotatebox{335}{\scriptsize{|}}}_-*!/d3pt/{\labelstyle{H}}^<<<<<<<{}="2" & {\cal I}\otimes {\cal I} \ar[r] \ar@{}[r]|-{\scriptsize{|}}^-*!/u3pt/{\labelstyle{H\otimes H}} & \duof\otimes\duof \ar[d] \ar@{}[d]|-{\rotatebox{90}{\scriptsize{|}}}^-{\, R}_-{}="1" &&  {\cal I} \otimes {\cal I} \ar[d] \ar@{}[d]|-{\rotatebox{90}{\scriptsize{|}}}_-{K\otimes K\,}^-{}="3" &\simeq& {\cal I} \ar[d] \ar@{}[d]|-{\rotatebox{90}{\scriptsize{|}}}^-{\, K}_-{}="4" \ar@{}"3";"4"|-{\Longrightarrow}^-{\mu} \\
& & \duof \ar@{}"1";"2"|<<<<<<<<<<{\Longrightarrow}_<<<<<<<<<<{\delta} && \duof\otimes\duof \ar[rr] \ar@{}[rr]|-{\scriptsize{|}}_-*!/d2pt/{\labelstyle{S}} && \duof
}$$
$$\xymatrix{
{\cal I} \ar@/^3ex/[rr] \ar@/^3ex/@{}[rr]|-{\scriptsize{|}}^-*!/u2pt/{\labelstyle{H}}_-{\phantom{.}}="1"  \ar@/_3ex/[rr] \ar@/_3ex/@{}[rr]|-{\scriptsize{|}}_-*!/d2pt/{\labelstyle{K}}^-{\phantom{.}}="2" \ar@{=>}"1";"2"^-{\tau}  && \duof
}$$
compatible with the two pseudomonoid structures.
By composition of $\ev$-modules these $2$-cells have component morphisms
$$\xymatrix{
\int^{X,Y}R(X;A,B)\otimes R(Y;C,D) \otimes S(E;X,Y) \ar[d]^{\gamma} \\ \int^{U,V}S(U;A,C)\otimes S(V;B,D)\otimes R(E;U,V)
}$$
$$\xymatrix{
H (A) \ar[rr]^-{\delta} && \int^{X,Y} H(X)\otimes H(Y)\otimes R(A;X,Y)
}$$
$$\xymatrix{
 \int^{X,Y} K(X)\otimes K(Y)\otimes S(A;X,Y) \ar[rr]^-{\mu} &&  K (A) 
}$$
$$\xymatrix{
H(A) \ar[rr]^-{\tau} && K(A)
}$$
in $\ev$.

Given any duoidal $\ev$-category $\duof$ we obtain a produoidal $\ev$-category structure on $\duof$ by setting
$$S(A;B,C) = \duof (A, B \ast C)$$
and
$$R(A;B,C) = \duof (A, B \circ C)$$
that is, we pre-compose the $\ev$-valued hom of $\duof$ with (\ref{starfunc}) and (\ref{circfunc}) of Definition \ref{duocatdef}.

\begin{proposition}
If $\duof$ is a produoidal $\ev$-category then $\bm ( \bid , \duof )$ = $[\duof\op , \ev]$ is a duoidal $\ev$-category.
\end{proposition}
\begin{proof}
Consider the $\ev$-category of $\ev$-functors and $\ev$-natural transformations $[\duof\op ,\ev]$.
The two monoidale structures on $\duof$ translate to two monoidal structures on $[\duof\op ,\ev]$ by Day-convolution
\begin{eqnarray}
(M\ast N)(A) & = & \int^{X,Y} S(A;X,Y)\otimes M(X)\otimes N(Y) \\
(M\circ N)(B) & = & \int^{U,V} R(B;U,V)\otimes M(U)\otimes N(V) 
\end{eqnarray}
such that the duoidale $2$-cell structure morphisms lift to give a duoidal $\ev$-category.
More specifically the maps $(\gamma , \delta , \mu , \tau )$ lift to $[\duof\op , \ev]$ and satisfy the axioms (\ref{gammaaxiom1}), (\ref{gammaaxiom2}), (\ref{deltaaxiom}) and (\ref{muaxiom}) in Definition \ref{duocatdef}.
Demonstrating the lifting and commutativity of the requisite axioms uses iterated applications of the $\ev$-enriched Yoneda lemma and Fubini's interchange theorem as in [\ref{KellyBook}].\qed
\end{proof}

Our final theorem for this section permits us to apply the theory of categories enriched in a duoidal $\ev$-category $\duof$ even if the monoidal structures on $\duof$ are not closed.

\begin{theorem}
Let $\duof$ be a duoidal $\ev$-category. The Yoneda embedding $y: \duof \longrightarrow [\duof\op , \ev]$ gives $[\duof\op , \ev]$ as the duoidal cocompletion of $\duof$ with both monoidal structures closed.
\end{theorem}
\begin{proof}
This theorem is essentially an extension of some results of Im and Kelly in [\ref{ImKelly}] which themselves are largely extensions of results in [\ref{DayThesis}] and [\ref{KellyBook}].
In particular, if $\ba$ is a monoidal $\ev$-category then $\hat{\ba} = [\ba\op , \ev]$ is the free monoidal closed completion with the convolution monoidal structure.
If $\duof$ is a duoidal $\ev$-category then, by Proposition 4.1 of [\ref{ImKelly}], the monoidal structures $\ast$ and $\circ$ on $\duof$ give two monoidal biclosed structures on $\hat{\duof} = [\duof\op , \ev ]$ with the corresponding Yoneda embeddings strong monoidal functors.
As per [\ref{ImKelly}] the monoidal products are given by Day convolution 
\begin{eqnarray}
P\,\,\hat{\ast}\,\, Q & = & \int^{A,B} P(A)\otimes Q(B) \otimes \duof (-, A\ast B) \\
P\,\,\hat{\circ}\,\, Q & = & \int^{A,B} P(A)\otimes Q(B) \otimes \duof (-, A\circ B)
\end{eqnarray}
as the left Kan-extension of $y\otimes y$ along the composites $y\ast$ and $y\circ$ respectively.
Write $\hat{J}$ and $\hat{\mathbf{1}}$ for the tensor units $y(J) = \duof (-,J)$ and $y(\mathbf{1}) = \duof (-, \mathbf{1} )$  respectively.
The duoidal data $(\gamma , \mu , \delta ,\tau)$ lifts directly to give duoidal data $(\hat{\gamma} , \hat{\mu} , \hat{\delta}, \hat{\tau})$ for $\duof$. \qed
\end{proof}

\section{Enrichment in a duoidal $\ev$-category base}\label{enrichduoidsection}

Let $\duof$ be a duoidal $\ev$-category.
There is a $2$-category $\duof_{h}$-Cat of $\duof_{h}$-categories, $\duof_{h}$-functors, and $\duof_{h}$-natural transformations in the usual Eilenberg-Kelly sense; see [\ref{KellyBook}].
We write ${\cal J}$ for the one-object $\duof_{h}$-category whose hom is the horizontal unit $J$ in $\duof$.

Let $\ba$ and $\bB$ be $\duof_{h}$-categories and define $\ba\circ\bB$ to be the $\duof_{h}$-category with objects pairs $(A,B)$ and hom-objects $(\ba\circ\bB )((A,B),(A',B')) = \ba (A,A') \circ \bB (B,B')$ in $\duof_{h}$.
Composition is defined using the middle of four map $\gamma$ as follows
$$\xymatrix{ 
(\ba\circ\bB )((A',B'),(A'',B'')) \ast (\ba\circ\bB )((A,B),(A',B')) \ar[d]_-{\cong} \\
(\ba (A',A'')\circ \bB (B',B'')) \ast (\ba (A,A') \circ \bB (B,B')) \ar[d]_-{\gamma} \\
(\ba (A',A'')\ast \ba (A,A')) \circ (\bB (B',B'') \ast \bB (B,B')) \ar[d]_-{comp\,\circ\, comp} \\
\ba (A,A'') \circ \bB (B,B'') \ar[d]_-{\cong} \\
(\ba\circ\bB )((A,B),(A'',B''))\, .
}$$
Identities are given by the composition
$$\xymatrix{J \ar[r]^-{\delta} & J\circ J \ar[rr]^-{\hat{id}_{A}\circ \hat{id}_{B}} &&  \ba (A,A) \circ \bB (B,B) \, .}$$
The monoidal unit is the $\duof_{h}$-category $\overline{\mathbf{1}}$ consisting of a single object $\bullet$ and hom-object $\overline{\mathbf{1}}(\bullet , \bullet) = \mathbf{1}$.

Checking the required coherence conditions proves the following result of [\ref{BatMarkl}].
\begin{proposition}
The $\circ$ monoidal structure on $\duof_{h}$ lifts to a monoidal structure on the $2$-category $\duof_{h}$-Cat.
\end{proposition}

We write $\duof$-Cat for the monoidal $2$-category $\duof_{h}$-Cat with $\circ$ as the tensor product.

Let $\duof$ be a duoidal $\ev$-category such that the horizontal monoidal structure $\ast$ is left-closed.
That is, we have
$$\duof (X\ast Y, Z) \cong \duof (X, [Y,Z])$$
with the ``evaluation'' counit $ev: [Y,Z]\ast Y \longrightarrow Z$.

This gives $\duof_{h}$ as an $\duof_{h}$-category in the usual way by defining the composition operation $[Y,Z]\ast [X,Y] \longrightarrow [X,Z]$ as corresponding to
$$\xymatrix{([Y,Z]\ast [X,Y]) \ast X \phantom{A} \cong \phantom{A} [Y,Z]\ast ([X,Y] \ast X) \ar[rr]^-{1\ast ev} && [Y,Z]\ast Y \ar[r]^-{ev} & Z}$$
and identities $\hat{id}_{X}: J \longrightarrow [X,X]$ as corresponding to $\ell : J\ast X\longrightarrow X$.

The duoidal structure of $\duof$ provides a way of defining $[X,X']\circ [Y,Y'] \longrightarrow [X\circ Y , X'\circ Y']$ using the the middle-of-four interchange map:
\begin{eqnarray}\xymatrix{([X,X']\circ [Y,Y'])\ast (X\circ Y) \ar[d]_-{\gamma} \ar[rr] && X' \circ Y' \\ ([X,X']\ast X) \circ ([Y,Y']\ast Y) \ar@/_/[rru]_-{ev\,\circ\, ev} &&}\label{duofhmulti}\end{eqnarray}

The above shows that $\duof$ is a monoidale (pseudo-monoid) in the category of $\duof_{h}$-categories with multiplication given by the $\duof_{h}$-functor $\hat{\circ} : \duof_{h}\circ\duof_{h} \longrightarrow\duof_{h}$ as defined.

Let Mon($\duof_{h}$) be the category of (horizontal) monoids $(M, \mu:M\ast M \longrightarrow M, \eta : J \longrightarrow M)$ in $\duof_{h}$.
Let $M$ and $N$ be objects of Mon($\duof_{h}$) and define the monoid multiplication map of $M\circ N$ to be the composition
$$\xymatrix{ (M\circ N) \ast (M\circ N) \ar[r]^-{\gamma} & (M\ast M) \circ (N\ast N) \ar[r]^-{\mu \circ \mu} & M\circ N}$$
and the unit to be
$$\xymatrix{J \ar[r]^-{\delta}& J\circ J \ar[rr]^-{\eta\circ\eta} && M\circ N}\, .$$

This tensor product of monoids is the restriction to one-object $\duof_{h}$-categories of the tensor of $\duof$-Cat.
So we have the following result which was also observed in [\ref{AgMaha}].

\begin{proposition}
The monoidal structure $\circ$ on $\duof$ lifts to a monoidal structure on the category Mon($\duof_{h}$).
\end{proposition}

We write Mon $\duof$ for the monoidal category Mon($\duof_{h}$) with $\circ$.

\begin{remark}
A monoid in $(\textrm{Mon}\,\duof )\op$ is precisely a bimonoid in $\duof$.
\end{remark}

\section{The Tannaka adjunction revisited}\label{revisitetannakaathome}

Let $\duof$ be a horizontally left closed duoidal $\ev$-category.
Each object $M$ of $\duof$ determines an $\duof_{h}$-functor 
$$\xymatrix{-\ast M : \duof_{h} \ar[r] & \duof_{h}}$$
defined on objects by $A\mapsto A\ast M$ and on homs by taking
\begin{eqnarray}
\xymatrix{-\ast M: [A,B] \ar[r] & [A\ast M, B\ast M]}
\end{eqnarray}
to correspond to 
$$\xymatrix{[A,B]\ast (A\ast M) \phantom{A}\cong\phantom{A} ([A,B]\ast A)\ast M \ar[rr]^-{ev\ast 1} && B\ast M}.$$
If $M$ is a monoid in $\duof_{h}$ then $-\ast M$ becomes a monad in $\duof_{h}$-Cat in the usual way.

We write $\duof^{\ast M}$ for the Eilenberg-Moore $\duof_{h}$-category of algebras for the $\duof_{h}$-monad $-\ast M$; see [\ref{LintonSLN}] and [\ref{FTM}].
It is the $\duof_{h}$-category of right $M$-modules in $\duof$.
If $\duof$ has equalizers then $\duof^{\ast M}$ is assured to exist; the $\duof_{h}$-valued hom is the equalizer of the pair
\begin{eqnarray} \label{effectonhoms}
\xymatrix{
[A,B] \ar[rr]^-{[\alpha , 1]} \ar[rd]_-{-\ast M} && [A\ast M,B] \\
& [A\ast M, B\ast M] \ar[ru]_-{[1,\beta ]} &
}
\end{eqnarray}
where $\alpha : A\ast M\longrightarrow A$ and $\beta : B\ast M \longrightarrow B$ are the actions of $A$ and $B$ as objects of $\duof^{\ast M}$.

Let $U_{M} : \duof^{\ast M} \longrightarrow \duof_{h}$ denote the underlying $\duof_{h}$-functor which forgets the action and whose effect on homs is the equalizer of (\ref{effectonhoms}).
There is an $\duof_{h}$-natural transformation
\begin{eqnarray}
\xymatrix{\chi : U_{M}\ast M \ar[r] & U_{M}}
\end{eqnarray}
which is the universal action of the monad $-\ast M$; its component at $A$ in $\duof^{\ast M}$ is precisely the action $\alpha : A\ast M\longrightarrow A$ of $A$.

An aspect of the strong enriched Yoneda Lemma is the $\duof_{h}$-natural isomorphism
\begin{eqnarray}
\duof^{\ast M} (M,B) \cong U_{M} B .
\end{eqnarray}
In this special case, the result comes from the equalizer
$$\xymatrix{
B \ar[r]^-{\hat{\beta}} & [M,B] \ar@<-1ex>[rr]_-{[1,\beta ](-\ast M)} \ar@<+1ex>[rr]^-{[\mu ,1]} && [M\ast M , B].
}$$
In other words, the $\duof_{h}$-functor $U_{M}$ is representable with $M$ as the representing object.

Each $\duof_{h}$-functor $U:\ba\longrightarrow\duof_{h}$ defines a functor 
\begin{eqnarray}
\xymatrix{U\ast - :\duof \ar[r] & \duof_{h}\textrm{-Cat}(\ba , \duof_{h})}
\end{eqnarray}
taking $X\in\duof$ to the composite $\duof_{h}$-functor 
$$\xymatrix{\ba \ar[rr]^-{U} && \duof_{h} \ar[rr]^-{-\ast X} && \duof_{h}}$$
and $f:X\longrightarrow Y$ to the $\duof_{h}$-natural transformation $U\ast f$ with components
$$\xymatrix{1\ast f : UA\ast X \ar[r] & UA\ast Y.}$$
We shall call $U:\ba\longrightarrow\duof_{h}$ \emph{tractable} when the functor $U\ast -$ has a right adjoint denoted
\begin{eqnarray}\label{rightadj}
\xymatrix{\{ U , - \} : \duof_{h}\textrm{-Cat}(\ba , \duof_{h}) \ar[r] & \duof.}
\end{eqnarray}
This means that morphisms $t : X\longrightarrow \{ U, V \}$ are in natural bijection with $\duof_{h}$-natural transformations $\theta : U\ast X\longrightarrow V$.

Let us examine what $\duof_{h}$-naturality of $\theta : U\ast X\longrightarrow V$ means.
By definition it means commutativity of 
\begin{eqnarray}
\xymatrix{
& [VA,VB] \ar[rd]^-{\phantom{.}[\theta_{A},1]} & \\
\ba (A,B) \ar[ru]^-{V_{A,B}} \ar[d]_-{U_{A,B}} && [UA\ast X , VB] \phantom{\, .}\\
 [UA,UB] \ar[rr]_-{-\ast X} && [UA\ast X , UB \ast X]  \ar[u]_-{[1,\theta_{B}]}\, .
}
\end{eqnarray}
This is equivalent to the module-morphism condition
\begin{eqnarray}\label{modulemorphcond}
\xymatrix{
\ba (A,B)\ast UA\ast X \ar[rr]^-{1\ast \theta_{A}} \ar[d]_-{\overline{U}_{A,B}} && \ba (A,B)\ast VA \ar[d]^-{\overline{V}_{A,B}} \\
UB\ast X \ar[rr]_-{\theta_{B}} && VB
}
\end{eqnarray}
under left closedness of $\duof_{h}$.
Notice that tractability of an object $Z$ of $\duof$, regarded as an $\duof_{h}$-functor $\ulcorner Z\urcorner : {\cal J}\longrightarrow \duof_{h}$, is equivalent to the existence of a horizontal right hom $\{ Z,- \}$:
\begin{eqnarray}
\duof (X,\{ Z,Y \} ) \cong \duof (Z\ast X, Y).
\end{eqnarray}
Assuming all of the objects $UA$ and $\ba (A,B)$ in $\duof$ are tractable, we can rewrite (\ref{modulemorphcond}) in the equivalent form
\begin{eqnarray}
\xymatrix{
& \{ UA,VA \} \ar[rr]^-{\{ 1, \hat{V}_{AB}\} } && \{ UA, \{ \ba (A,B),VB \} \} \ar[dd]^-{\labelstyle{\cong}} \\
X \ar[ru]^-{\hat{\theta}_{A}} \ar[dr]_-{\hat{\theta}_B} &&& \\
& \{ UB,VB \} \ar[rr]^-{\{ \hat{U}_{AB} , 1 \} } && \{ \ba (A,B) \ast UA, VB \} .
}
\end{eqnarray}

\begin{proposition}
If $\duof$ is a complete, horizontally left and right closed, duoidal $\ev$-category and $\ba$ is a small $\duof_{h}$-category then every $\duof_{h}$-functor $U: \ba\longrightarrow\duof_{h}$ is tractable.
\end{proposition}

However, some $U$ can still be tractable even when $\ba$ is not small.

\begin{proposition}\label{yonedaprop}
(Yoneda Lemma) If $U:\ba\longrightarrow\duof_{h}$ is an $\duof_{h}$-functor represented by an object $K$ of $\ba$ then $U$ is tractable and
$$\{ U ,V \} \cong VK .$$
\end{proposition}
\begin{proof}
By the ``weak Yoneda Lemma'' (see [\ref{KellyBook}]) we have
$$\duof_{h}\textrm{-Cat}(U\ast X,V)\cong\duof_{h}\textrm{-Cat}(U,[X,V])\cong\duof (J,[X,VK])\cong\duof (X,VK). \qed$$
\end{proof}

Consider the $2$-category $\duof_{h}\textrm{-Cat}\downarrow\ps\duof_{h}$ defined as follows.
The objects are $\duof_{h}$-functors  $U:\ba\longrightarrow\duof_{h}$.
The morphisms $(T,\tau ):U\longrightarrow V$ are triangles 
\begin{eqnarray}
\xymatrix{
\ba \ar[rr]^-{T} \ar[rd]_-{U}^{}="1" &&  \bB \ar[ld]^-{V}_{}="2" \ar@{}"1";"2"|-{\cong}^*!/d2pt/-{\labelstyle{\tau}} \\
& \duof_{h} &
}
\end{eqnarray}
in $\duof_{h}$-Cat.
The $2$-cells $\theta : (T,\tau )\Longrightarrow (S,\sigma )$ are $\duof_{h}$-natural transformations $\theta :T\Longrightarrow S$ such that
\begin{eqnarray}
\xymatrix{
\ba \ar[rr]_-{S}^-*!/d1pt/{\phantom{.}}="3" \ar@/^5ex/[rr]^-{T}_-*!/u1pt/{\phantom{.}}="4" \ar@{=>}"4";"3"^-{\phantom{.}\theta} \ar[rd]_-{U}^{}="1" &&  \bB \ar[ld]^-{V}_{}="2" \ar@{}"1";"2"|-{\cong}^*!/d1pt/-{\labelstyle{\sigma}} \\
& \duof_{h} &
} & \begin{array}{c} \\ \\ \\  = \end{array} & \xymatrix{
\ba \ar@/^3ex/[rr]^-{T} \ar[rd]_-{U}^<<<{}="1" &&  \bB \ar[ld]^-{V}_<<<{}="2" \ar@{}"1";"2"|-{\cong}^*!/d2pt/-{\labelstyle{\tau}} \\
& \duof_{h} &.
} 
\end{eqnarray}
We define a \emph{vertical tensor product} $\underline{\circ}$ on the $2$-category $\duof_{h}\textrm{-Cat}\downarrow\ps\duof_{h}$ making it a monoidal $2$-category, which we denote by $\duof\textrm{-Cat}\downarrow\ps\duof$.
For $\duof_{h}$-functors $U : \ba\longrightarrow\duof_{h}$ and $V: \bB\longrightarrow\duof_{h}$, define $U\underline{\circ} V : \ba\circ\bB\longrightarrow\duof_{h}$ to be the composite
\begin{eqnarray}
\xymatrix{\ba\circ\bB \ar[rr]^-{U\circ V} && \duof_{h}\circ\duof_{h} \ar[r]^-{\hat{\circ}} & \duof_{h} }.
\end{eqnarray}
The unit object is $\ulcorner \mathbf {1} \urcorner : \overline{\mathbf{1}} \longrightarrow \duof_{h}$.
The associativity constraints are explained by the diagram
\begin{eqnarray}
\xymatrix{
(\ba\circ\bB )\circ \bc \ar[rr]^-{\cong} \ar[d]_-{(U\circ V ) \circ W} && \ba\circ (\bB\circ \bc ) \ar[d]^-{U\circ (V \circ W ) } \\
(\duof_{h} \circ \duof_{h} ) \circ \duof_{h} \ar[rr]^-{\cong} \ar[d]_-{\hat{\circ}\, \circ \, 1}^>{}="1" && \duof_{h} \circ ( \duof_{h} \circ \duof_{h} ) \ar[d]^-{1\,\circ\,\hat{\circ}}_>{}="2" \ar@{}"1";"2"|-{\cong}^-*!/u2pt/{\labelstyle{a}} \\
\duof_{h} \circ \duof_{h} \ar[rd]_-{\hat{\circ}} && \duof_{h} \circ \duof_{h} \ar[ld]^-{\hat{\circ}} \\
& \duof_{h} &
}
\end{eqnarray}
where $a$ is the associativity constraint for the vertical structure on $\duof$.
The unit constraints are similar.

\begin{remark}
We would like to emphasise that, although there are conceivable $2$-cells for Mon $\duof$ as a sub-$2$-category of $\duof_{h}$-Cat (see [\ref{FTM}]), we are only regarding Mon $\duof$ as a monoidal category, not a monoidal $2$-category.
\end{remark}

Next we specify a monoidal functor 
\begin{eqnarray}
\xymatrix{ \textrm{mod} :  (\textrm{Mon }\duof )\op \ar[r] & \duof\textrm{-Cat}\downarrow\ps\duof}.
\end{eqnarray}
For each monoid $M$ in $\duof_{h}$, we put
$$\textrm{mod } M = ( U_{M} : \duof^{\ast M} \xymatrix{\ar[r]&} \duof_{h} ) .$$
For a monoid morphism $f: N \longrightarrow M$, we define
\begin{eqnarray}
\xymatrix{
\duof^{\ast M} \ar[rr]^-{\textrm{mod }f} \ar[dr]_-{U_{M}} & \ar@{}[d]|-{=} & \duof^{\ast N} \ar[dl]^-{U_{N}} \\
& \duof & 
}
\end{eqnarray}
by
$$(\textrm{mod }f) (A\ast M\xymatrix{\ar[r]^-{\alpha}&} A ) = (A\ast N\xymatrix{\ar[r]^-{1\ast f}&} A\ast M \xymatrix{\ar[r]^-{\alpha}&} A). $$
To see that mod $f$ is an $\duof_{h}$-functor, we recall the equalizer of (\ref{effectonhoms}) and point to the following diagram in which the empty regions commute.
$$\scalebox{0.9}{\xymatrix{
&& [A\ast M,B] \ar@/^3ex/[rrd]^-{[1\ast f , 1]}  && \\
[A,B] \ar@/^3ex/[rru]^-{[\alpha , 1]} \ar@{}[rru]_->>>>>>>>>>>>{(\ref{effectonhoms})\phantom{....}} \ar[rr]_-{-\ast M} \ar[dr]_-{-\ast N} && [A\ast M, B\ast M] \ar[u]_-{[1,\beta ]} \ar[rd]^-{[1\ast f,1]} && [A\ast N, B] \\
& [A\ast N, B\ast N] \ar[rr]_-{[1,1\ast f]} && [A\ast N, B\ast M] \ar[ru]_-{[1,\beta ]} &
}}$$
Alternatively, we could use the universal property of mod $N$ as the universal action of the monad $-\ast N$ on $\duof$.

For the monoidal structure on mod, we define an $\duof_{h}$-functor $\Phi_{M,N}$ making the square
\begin{eqnarray}\label{modismonoidal}
\xymatrix{
\duof^{\ast M}\circ\duof^{\ast N} \ar[rr]^-{\Phi_{M,N}} \ar[d]_-{U_{M}\circ U_{N}} && \duof^{\ast (M\circ N)} \ar[d]^-{U_{M\circ N}} \\
\duof_{h}\circ\duof_{h} \ar[rr]_-{\hat{\circ}} && \duof_{h}
}
\end{eqnarray}
commute; put
$$\begin{array}{lc}\Phi_{M,N} (A\ast M\xymatrix{\ar[r]^-{\alpha}&} A, B\ast N\xymatrix{\ar[r]^-{f}&} B) = \phantom{AAAAAAAAAAAAAAA}&\end{array} $$
$$\begin{array}{cr} \phantom{AAAAAAAA} & ((A\circ B)\ast (M\circ N) \xymatrix{\ar[r]^-{\gamma}&} (A\ast M )\circ (B\ast N) \xymatrix{\ar[r]^-{\alpha\circ\beta}&} A\circ B) \end{array}$$
and use the universal property of mod$(M\circ N)$ to define $\Phi_{M,N}$ on homs.
\newline

For tractable $U:\ba\longrightarrow\duof_{h}$, we have an evaluation $\duof_{h}$-natural transformation
$$ev : U\ast \{ U,V \} \xymatrix{\ar[r]&} V ,$$
corresponding under the adjunction (\ref{rightadj}), to the identity of $\{ U, V \}$.
We have a ``composition morphism''
$$\xymatrix{\mu: \{ U,V \} \ast \{ V, W\} \ar[r] & \{ U,W \} }$$
corresponding to the composite 
$$\xymatrix{ U \ast \{ U, V\} \ast \{ V, W \} \ar[rr]^-{ev\ast 1} && V\ast \{ V, W\} \ar[r]^-{ev} & W. }$$
In particular,
$$\xymatrix{\mu: \{ U,U \} \ast \{ U, U\} \ar[r] & \{ U,U \} }$$
together with 
$$\xymatrix{\eta : J \ar[r] & \{ U,U \}},$$
corresponding to $U\ast J\cong U$, gives $\{ U,U \}$ the structure of a monoid, denoted end $U$, in $\duof_{h}$.

\begin{proposition}
For each tractable $\duof_{h}$-functor $U:\ba\longrightarrow\duof_{h}$, there is an equivalence of categories
$$(\mathrm{Mon}\,\duof_{h})(M,\mathrm{end}\, U) \simeq (\duof_{h} \textrm{-}\mathrm{Cat}\downarrow\ps\duof_{h})(U, \mathrm{mod}\, M)$$
pseudonatural in monoids $M$ in $\duof_{h}$.
\end{proposition}
\begin{proof}
Morphisms $t:M\longrightarrow\textrm{end }U$ in $\duof$ are in natural bijection (using (\ref{rightadj})) with $\duof_{h}$-natural transformations $\theta : U\ast M\longrightarrow U$.
It is easy to see that $t$ is a monoid morphism if and only if $\theta$ is an action of the monad $-\ast M$ on $U:\ba\longrightarrow\duof_{h}$.
By the universal property of the Eilenberg-Moore construction [\ref{FTM}], such actions are in natural bijection with liftings of $U$ to $\duof_{h}$-functors $\ba\longrightarrow\duof^{\ast M}$.
This describes a bijection between $(\textrm{Mon }\duof_{h})(M,\textrm{end }U)$ and the full subcategory of $(\duof_{h}\textrm{-Cat}\downarrow\ps\duof_{h})(U, \textrm{mod }M)$ consisting of the morphisms
$$
\xymatrix{
\ba \ar[rr]^-{T} \ar[rd]_-{U}^{}="1" &&  \duof^{\ast M} \ar[ld]^-{U_{M}}_{}="2" \ar@{}"1";"2"|-{\cong}^*!/d2pt/-{\labelstyle{\tau}} \\
& \duof_{h} &
}
$$
for which $\tau$ is an identity.
It remains to show that every general such morphism $(T,\tau )$ is isomorphic to one for which $\tau$ is an identity.
However, each $(T,\tau )$ determines an action
$$\xymatrix{U\ast M \ar@{}[r]|-{\cong}^-*!/u2pt/{\labelstyle{\tau}} & U_{M}T\ast M = (U_{M}\ast M)T \ar[r]^-{\chi T} & U_{M}T \ar@{}[r]|-{\cong}^-*!/u2pt/{\labelstyle{\tau^{-1}}} & U}$$
of the monad $-\ast M$ on $U$.
By the universal property, we induce a morphism
$$
\xymatrix{
\ba \ar[rr]^-{T'} \ar[rd]_-{U}^>>>>>>>>>{}="1" &&  \duof^{\ast M} \ar[ld]^-{U_{M}}_ >>>>>>>>>{}="2" \ar@{}"1";"2"|-{=} \\
& \duof_{h} &
}
$$
and an invertible $2$-cell $(T,\tau )\cong (T',1)$ in $\duof_{h}\textrm{-Cat}\downarrow\ps\duof_{h}$.\qed
\end{proof}

In other words, we have a biadjunction 
\begin{eqnarray}\label{biadjunc}
\xymatrix{(\textrm{Mon }\duof_{h})\op \ar@<-1.5ex>[rr]_-{\textrm{mod}}="1" && \ar@<-1.5ex>[ll]_-{\textrm{end}}="2" \ar@{}"1";"2"|-{\perp} \duof_{h}\textrm{-Cat}\downarrow\ps\tract\duof_{h}}
\end{eqnarray}
where the $2$-category on the right has objects restricted to the tractable $U$.
As a consequence, notice that end takes each $2$-cell to an identity (since all $2$-cells in Mod $\duof_{h}$ are identities).
Notice too from the notation that we are ignoring the monoidal structure in (\ref{biadjunc}).
This is because tractable $U$ are not generally closed under the monoidal structure of $\duof\textrm{-Cat}\downarrow\ps\duof$.

\begin{proposition}
Representable objects of $\duof\textrm{-Cat}\downarrow\ps\duof$ are closed under the monoidal structure.
\end{proposition}
\begin{proof}
$$\xymatrix{\ba\circ\bB \ar[rrr]^-{\ba (A,-)\circ\bB (B,-)} \ar@<-1ex>@/_4ex/[rrrr]_-{(\ba\circ\bB )((A,B),-)} \ar@{}@<-2.5ex>[rrrr]|{\cong} &&& \duof_{h}\circ\duof_{h} \ar[r]^-{\hat{o}} & \duof_{h}}\, .$$
and
$$\xymatrix{\ulcorner \mathbf{1} \urcorner = \overline{\mathbf{1}} (\bullet , -): \overline{\mathbf{1}} \ar[r] & \duof_{h}}.\qed$$
\end{proof}

Let $\duof\textrm{-Cat}\downarrow\ps\rep\duof$ denote the monoidal full sub-$2$-category of $\duof\textrm{-Cat}\downarrow\ps\duof$ consisting of the representable objects.
The biadjunction (\ref{biadjunc}) restricts to a biadjunction
\begin{eqnarray}\label{biadjunc2}
\xymatrix{(\textrm{Mon }\duof_{h})\op \ar@<-1.5ex>[rr]_-{\textrm{mod}}="1" && \ar@<-1.5ex>[ll]_-{\textrm{end}}="2" \ar@{}"1";"2"|-{\perp} \duof_{h}\textrm{-Cat}\downarrow\ps\rep\duof_{h}}
\end{eqnarray}
and we have already pointed out that mod is monoidal; see (\ref{modismonoidal}).
In fact, we shall soon see that this is a monoidal biadjunction.

First note that, if $U:\ba\longrightarrow\duof_{h}$ is represented by $K$ then we have a monoidal isomorphism
\begin{eqnarray}\label{isoreduce}
\xymatrix{\textrm{end }U = \{ U,U \} \ar@{}[r]|-{\cong}^-*!/u2pt/{\labelstyle{(\ref{yonedaprop})}} & UK \cong \ba (K,K).}
\end{eqnarray}
In particular, for a monoid $M$ in $\duof_{h}$, using Proposition \ref{canonicalmonoidal}, we obtain a monoid isomorphism
\begin{eqnarray}\label{endmodmism}
\textrm{end mod }M\cong M
\end{eqnarray}
which is in fact the counit for (\ref{biadjunc2}), confirming that mod is an equivalence on homs.

\begin{proposition}\label{endisstrongmonoidal}
The $2$-functor $\mathrm{end}$ in  (\ref{biadjunc2}) is strong monoidal.
\end{proposition}
\begin{proof}
The isomorphism (\ref{isoreduce}) gives
\begin{eqnarray*}
\textrm{end }(\ba\circ\bB )((A,B),-) & \cong & (\ba\circ\bB )((A,B),(A,B)) \\
& \cong & \ba (A,A) \circ \bB (B,B)  \\
& \cong & \textrm{end }\ba (A,-) \circ \textrm{end }\bB (B,-)
\end{eqnarray*}
and
$$\textrm{end }\overline{\mathbf{1}}(\bullet , -) \cong \overline{\mathbf{1}}(\bullet , \bullet ) \cong \mathbf{1}. \qed  $$
\end{proof}

As previously remarked, a monoid in $(\textrm{Mon }\duof_{h})\op$ is precisely a bimonoid in $\duof$; see Definition \ref{bimonoidinf}.
Since Mon $\duof$ has discrete homs, these monoids are the same as pseudomonoids.
The biadjunction (\ref{biadjunc}) determines a biadjunction
\begin{eqnarray}\label{biadjunc3}
\xymatrix{(\textrm{Bimon }\duof_{h})\op \ar@<-1.5ex>[rr]_-{\textrm{mod}}="1" && \ar@<-1.5ex>[ll]_-{\textrm{end}}="2" \ar@{}"1";"2"|-{\perp} \textrm{Mon}\lowps(\duof\textrm{-Cat}\downarrow\ps\rep\duof ).}
\end{eqnarray}
A pseudomonoid in $\duof\textrm{-Cat}\downarrow\duof$ is a monoidal $\duof_{h}$-category $\ba$ together with a strong monoidal $\duof_{h}$-functor $U:\ba\longrightarrow\duof_{h}$ (where $\duof_{h}$ has $\hat{\circ}$ as the monoidal structure).

This leads to the following lifting to the duoidal setting of a result attributed to Bodo Pareigis (see [\ref{Bodo1}], [\ref{Bodo2}] and [\ref{Bodo3}]).

\begin{theorem}\label{bimonoidisoclassmonoidal}
For a horizontal monoid $M$ in a duoidal $\ev$-category $\duof$, bimonoid structures on $M$ are in bijection with isomorphism classes of monoidal structures on $\duof^{\ast M}$ such that $U_{M}:\duof^{\ast M}\longrightarrow\duof$ is strong monoidal into the vertical structure on $\duof$.
\end{theorem}
\begin{proof}
For any horizontal monoid $M$ in $\duof$ we (in the order they appear) have (\ref{biadjunc}), Proposition \ref{endisstrongmonoidal} and (\ref{endmodmism}) giving
$$\begin{array}{rl}
& (\duof\textrm{-Cat}\downarrow\ps\duof )(\textrm{mod }M\circ\textrm{mod }M, \textrm{mod }M)\\
\cong & (\textrm{Mon }\duof )(M, \textrm{end }(\textrm{mod }M \circ \textrm{mod }M)) \\
\cong & (\textrm{Mon }\duof )(M, \textrm{end mod }M \circ \textrm{end mod }M)) \\
\cong & (\textrm{Mon }\duof )(M, M\circ M)
\end{array}$$
and
$$ (\duof\textrm{-Cat}\downarrow\ps\duof )( \ulcorner\mathbf{1}\urcorner , \textrm{mod }M) \simeq (\textrm{Mon }\duof )(M,\mathbf{1} ).$$
By Proposition \ref{endisstrongmonoidal}, each bimonoid structure on $M$ yields a pseudomonoid structure on mod $M$; and each pseudomonoid structure on mod $M$ yields a bimonoid structure on end mod $M \cong M$.
The above equivalences give the bijection of the Theorem.\qed
\end{proof}

\section{Hopf bimonoids}\label{hopfbimonoids}

We have seen that a bimonoid $M$ in a duoidal $\ev$-category $\duof$ leads to a monoidal $\duof_{h}$-category $\duof^{\ast M}$ of right $M$-modules.
In this section, we are interested in when $\duof^{\ast M}$ is closed.
We lean heavily on papers [\ref{BLV}] and [\ref{CLS}].

A few preliminaries from [\ref{StreetPub88}] adapted to $\duof_{h}$-categories are required.
For an $\duof_{h}$-category $\ba$, a \emph{right $\ba$-module} $\xymatrix{W: {\cal J} \ar[r]|-*=0{\labelstyle{|}} & \ba}$ is a family of objects $WA$ of $\duof$ indexed by the objects $A$ of $\ba$ and a family 
$$\xymatrix{W_{AB} : WA \ast \ba (B,A) \ar[r] & WB}$$
of morphisms of $\duof$ indexed by pairs of objects $A$, $B$ of $\ba$, satisfying the action conditions.
For modules $\xymatrix{W, W': {\cal J} \ar[r]|-*=0{\labelstyle{|}} & \ba}$, define $[W,W']$ to be the limit as below when it exists in $\duof$.
\begin{eqnarray}
\xymatrix{
& [WA,W'A] \ar[rr]^-{-\ast\ba (B,A)} && [WA\ast\ba (B,A) , W'A\ast\ba (B,A)] \ar[dd]^-{[1,W'_{AB}]} \\
[W,W'] \ar@{-->}[ur] \ar@{-->}[dr] && \\
& [WB,W'B] \ar[rr]_-{[W_{AB}, 1]} && [WA\ast\ba (B,A), W'B]
}
\end{eqnarray}

\begin{example}
A monoid $M$ in $\duof_{h}$ can be regarded as a one object $\duof_{h}$-category.
A right $M$-module $\xymatrix{A: {\cal J} \ar[r]|-*=0{\labelstyle{|}} & M}$ is precisely an object of $\duof^{\ast M}$.
\end{example}

\begin{example}
For any $\duof_{h}$-functor $S:\ba\longrightarrow\bx$ and object $X$ of $\bx$, we obtain a right $\ba$-module $\xymatrix{\bx (S,X): {\cal J} \ar[r]|-*=0{\labelstyle{|}} & \ba}$ defined by the objects $\bx (SA,X)$ of $\duof$ and the morphisms
$$\xymatrix{\bx (SA,X) \ast \ba (B,A) \ar[r]^-{1\ast S_{BA}} & \bx (SA,X) \ast \bx (SB,SA) \ar[r]^-{comp} & \bx (SB,X)}.$$
\end{example}

Recall from [\ref{StreetPub88}] that the \emph{colimit }$\textrm{colim}(W,S)$\emph{ of $S:\ba\longrightarrow\bx$ weighted by $\xymatrix{W: {\cal J} \ar[r]|-*=0{\labelstyle{|}} & \ba}$} is an object of $\bx$ for which there is an $\duof_{h}$-natural isomorphism
\begin{eqnarray}\label{weightedcolim}
\bx (\textrm{colim}(W,S),X) \cong [W,\bx (S,X)]
\end{eqnarray}
By Yoneda, such an isomorphism is induced by the module morphism
\begin{eqnarray}
\xymatrix{\lambda : W \ar[r] & \bx (S,\textrm{colim}(W,S))}.
\end{eqnarray}
The $\duof_{h}$-functor $S:\ba\longrightarrow\bx$ is \emph{dense} when $\lambda = 1 : \bx (S,Y) \longrightarrow \bx (S,Y)$ induces
\begin{eqnarray}\label{densebycolim}
\textrm{colim}(\bx (S,Y),S) \cong Y
\end{eqnarray}
for all $Y$ in $\bx$.

\begin{proposition}\label{cornerjisdense}
The $\duof_{h}$-functor $\ulcorner J \urcorner : {\cal J} \longrightarrow\duof_{h}$ is dense.
\end{proposition}
\begin{proof}
From (\ref{weightedcolim}) we see that 
$$[Y,X] \cong [J, [Y,X]]$$
implies 
$$\textrm{colim}([J,Y],J)\cong Y,$$
which is (\ref{densebycolim}) in this case.\qed
\end{proof}

Another element of our analysis is to recast the middle-of-four interchange morphisms as a $2$-cell in $\duof_{h}$-Cat.

\begin{proposition}\label{gammaisatwocell}
The family of morphisms 
$$\xymatrix{\gamma : (X\circ Y )\ast (C \circ D)  \ar[rr] && (X\ast C) \circ (Y\ast D)}$$
defines an $\duof_{h}$-natural transformation
$$\xymatrix{
\ar@{}[rrrr]_<<<<<<<<<<<<<<<<<<<<<<<<<<<<*!/d12pt/{\labelstyle{\phantom{A}}}="1" && \duof \ar[rrd]^-{-\ast (C\circ D)} && \\
\duof\circ\duof \ar[rru]^-{\hat{\circ}} \ar[rrd]_-{(-\ast C)\circ (-\ast D)\phantom{AA}} &&&& \duof \\
\ar@{}[rrrr]^<<<<<<<<<<<<<<<<<<<<<<<<<<<<*!/u12pt/{\labelstyle{\phantom{A}}}="2" \ar@{=>}"1";"2"^-{\phantom{a}\gamma} && \duof\circ\duof \ar[rru]_-{\hat{\circ}} && 
}$$
for all objects $C$ and $D$ of $\duof$.
\end{proposition}
\begin{proof}
Regard the commutative diagram
$$\xymatrix{
([X,U]\circ [Y,V] ) \ast ((X\ast C)\circ (Y\ast D)) \ar[d]_-{\gamma}="1" & \ar[l]_-{1\ast\gamma} ([X,U]\circ [Y,V] )\ast (X\circ Y)\ast (C\circ D) \ar[d]^-{\gamma\ast 1}="2" \ar@{}"1";"2"|-*=0{(\ref{gammaaxiom1})}  \\
([X,U]\ast X\ast C)\circ ([Y,V]\ast Y\ast D) \ar[d]_-{(ev\ast C)\circ (ev\ast D)}="1" & \ar[l]^-{\gamma}  ( ([U,X]\ast X)\circ ([Y,V]\ast Y) ) \ast (C\circ D) \ar[d]^-{(ev\circ ev)\ast 1}="2" \ar@{}"1";"2"|-*=0{\labelstyle{naturality}} \\
 (U\ast C)\circ (V\ast D) &  (U\circ V)\ast (C\circ D) \ar[l]^-{\gamma}
}$$
in which we have written as if $\ast$ were strict.\qed
\end{proof}

\begin{proposition}\label{thetacolimisinv}
Suppose $\theta: F \Longrightarrow G: \bx \longrightarrow \by$ is an $\duof_{h}$-natural transformation between $\duof_{h}$-functors $F$ and $G$ which preserve colimits weighted by $\xymatrix{W: {\cal J} \ar[r]|-*=0{\labelstyle{|}} & \ba}$.
If each $\theta_{SA}:FSA\longrightarrow GSA$ is invertible then so is
$$\xymatrix{\theta_{\mathrm{colim}(W,S)} : F\, \mathrm{colim}(W,S) \ar[r] & G\, \mathrm{colim}(W,S)}.$$
\end{proposition}
\begin{proof}
$$\xymatrix{
F\,\textrm{colim}(W,S) \ar[rr]^-{\theta_{\textrm{colim}(W,S)}} \ar[d]_-{\cong} && G\,\textrm{colim}(W,S) \ar[d]^-{\cong}\phantom{.} &\\
 \textrm{colim}(W,FS) \ar[rr]_-{\theta_{\textrm{colim}(1,\theta_{S})}} &&  \textrm{colim}(W,GS) .& \qed
}$$
\end{proof}

\begin{definition}
For a bimonoid $M$ in a duoidal category $\duof$, the composite $v_{\ell}$:
$$\xymatrix{(J\circ M)\ast M \ar[r]^-{1\ast\delta} & (J\circ M)\ast (M\circ M) \ar[r]^-{\gamma} & (J\ast M) \circ (M\ast M) \ar[r]^-{\ell\circ\mu} & M\circ M}$$
is called the \emph{left fusion morphism}.
The composite $v_{r}$:
$$\xymatrix{(M\circ J)\ast M \ar[r]^-{1\ast\delta} & (M\circ J)\ast (M\circ M) \ar[r]^-{\gamma} & (M\ast M) \circ (J\ast M) \ar[r]^-{\mu\circ\ell} & M\circ M}$$
is called the \emph{right fusion morphism}.
We call $M$ \emph{left Hopf} when $v_{\ell}$ is invertible and \emph{right Hopf} when $v_{r}$ is invertible.
We call $M$ \emph{Hopf} when both $v_{\ell}$ and $v_{r}$ are invertible.
\end{definition}

Suppose $\ba$ and $\bx$ are monoidal $\duof_{h}$-categories and $U:\ba\longrightarrow\bx$ is a monoidal $\duof_{h}$-functor.
Writing $\circ$ for the tensor and $\mathbf{1}$ for the tensor unit, we must have morphisms
$$\xymatrix{\varphi : UA\circ UB \ar[r] & U (A\circ B)  & \textrm{and} & \varphi_{0} : \mathbf{1} \ar[r] & U\mathbf{1}}$$
satisfying the usual Eilenberg-Kelly [\ref{SammyMax}] conditions.
Suppose $\ba$ and $\bx$ are left closed and write $\lom (A,B)$ and $\lom (X,Y)$ for the left homs.
As pointed out by Eilenberg-Kelly, the monoidal structure $\varphi$, $\varphi_{0}$ is in bijection with \emph{left closed structure}
$$\xymatrix{\varphi^{\ell} : U\lom (A,B)\circ UB \ar[r] & U \lom (UA,UB)  & \textrm{and} & \varphi_{0} : \mathbf{1} \ar[r] & U\mathbf{1},}$$
where $\varphi^{\ell}$ corresponds under the adjunction to the composite
$$\xymatrix{U\, hom (A,B)\circ UA \ar[r]^{\varphi} & U(hom(A,B)\circ A) \ar[r]^-{U_{ev}} & UB}$$
Following [\ref{DSquantum}], we say $U$ is \emph{strong left closed} when both $\varphi^{\ell}$ and $\varphi_{0}$ are invertible.

Recall from [\ref{BLV}] (and [\ref{CLS}] for the enriched situation) that the Eilenberg-Moore (enriched) category for an opmonoidal monad $T$ on $\bx$ is left closed and the forgetful $U_{T} : \bx^{T} \longrightarrow \bx$ is strong left closed if and only if $T$ is ``left Hopf''.
The monad $T$ is \emph{left Hopf} when the \emph{left fusion morphism}
\begin{eqnarray}
\xymatrix{v_{\ell} (X,Y) : T(X\circ TY) \ar[r]^-{\varphi} & TX\circ T^{2}Y \ar[r]^-{1\circ\mu} & TX\circ TY }
\end{eqnarray}
is invertible for all $X$ and $Y$.
It is \emph{right Hopf} when the \emph{right fusion morphism}
\begin{eqnarray}
\xymatrix{v_{r} (X,Y) : T(TX\circ Y) \ar[r]^-{\varphi} & T^{2}X\circ TY \ar[r]^-{\mu\circ 1} & TX\circ TY }
\end{eqnarray}
is invertible.

In particular, for a bimonoid $M$ in $\duof$, taking $T=-\ast M$, we see that $v_{\ell}(X,Y)$ is the composite
\begin{eqnarray}\label{lefthopfv}\xymatrix{
(X\circ (Y\ast M))\ast M \ar[dd]_-{v_{\ell}(X,Y)} \ar[rr]^-{1\ast\delta} && (X\circ (Y\ast M))\ast (M\circ M) \ar[d]^-{\gamma}  \\
&& (X\ast M)\circ ((Y\ast M)\ast M) \ar[d]^-{1\circ a}_-{\cong} \\
(X\ast M)\circ (Y\ast M) && \ar[ll]^-{1\circ (1\ast \mu)} (X\ast M) \circ (Y\ast (M\ast M)) 
}\end{eqnarray}
and that $v_{r}(X,Y)$ is
\begin{eqnarray}\label{righthopfv}\xymatrix{
((X\ast M) \circ Y)\ast M \ar[dd]_-{v_{r}(X,Y)} \ar[rr]^-{1\ast\delta} && ((X\ast M)\circ Y)\ast (M\circ M) \ar[d]^-{\gamma}  \\
&& ((X\ast M)\ast M)\circ (Y\ast M) \ar[d]^-{ a\circ 1}_-{\cong} \\
(X\ast M)\circ(Y\ast M) && \ar[ll]^-{(1\ast \mu) \circ 1} (X\ast (M\ast M))\circ (Y\ast M) .
}\end{eqnarray}

Recall from Section \ref{enrichduoidsection} that, when $\duof$ is horizontally left closed, not only does it become an $\duof_{h}$-category, it becomes a pseudomonoid in $\duof_{h}$-Cat using the tensor $\hat{\circ}$.
That is, $(\duof , \hat{\circ} ,  \ulcorner\mathbf{1}\urcorner )$ is a monoidal $\duof_{h}$-category.

We are interested in when $(\duof , \hat{\circ} ,  \ulcorner\mathbf{1}\urcorner )$ is closed and when the closed structure lifts to $\duof^{\ast M}$ for a bimonoid $M$ in $\duof$.

\begin{proposition}\label{closureprop}
The monoidal $\duof_{h}$-category $(\duof , \hat{\circ} ,  \ulcorner\mathbf{1}\urcorner )$  is closed if and only if
\begin{enumerate}
\item[(i)] $\duof_{v}$ is a closed monoidal $\ev$-category, and 
\item[(ii)] there exist $\ev$-natural isomorphisms
\end{enumerate}
$$X\circ (W\ast Y) \cong W\ast (X\circ Y) \cong (W\ast X)\circ Y .$$
\end{proposition}

\begin{proof}
To say  $(\duof , \hat{\circ} ,  \ulcorner\mathbf{1}\urcorner )$ is left closed is to say we have a ``left hom'' $\lom (X,Y)$ and an $\duof_{h}$-natural isomorphism
$$[X\circ Y,Z]\cong [X,\lom (Y,Z)].$$
By Yoneda, this amounts to a $\ev$-natural isomorphism
$$\duof (W, [X\circ Y,Z]) \cong \duof (W,[X,\lom (Y,Z)]).$$
Since $[\cdot,\cdot]$ is the horizontal left hom for $\duof$, this amounts to
\begin{eqnarray}\label{lomislefthom}\duof (W\ast (X\circ Y),Z) \cong \duof (W\ast X, \lom (Y,Z)).\end{eqnarray}
Taking $W=J$, we obtain
$$\duof (X\circ Y,Z) \cong \duof (X,\lom (Y,Z)),$$
showing that $\lom$ is a left hom for $\duof_{v}$ as a monoidal $\ev$-category.
So (i) is implied.
Now we have this, we can rewrite (\ref{lomislefthom}) as
$$\duof (W\ast (X\circ Y), Z) \cong \duof ((W\ast X)\circ Y, Z)$$
which, again by Yoneda, is equivalent to 
\begin{eqnarray}W\ast (X\circ Y) \cong (W\ast X)\circ Y. \end{eqnarray}
Similarly, to say $(\duof , \hat{\circ} , \mathbf{1})$ is right closed means
$$[X\circ Y, Z] \cong [ Y,\rom (X,Z)], $$
which means 
$$\duof (W\ast (X\circ Y) , Z) \cong \duof (W\ast Y, \rom (X,Z)).$$
Taking $W=J$, we see that $\rom$ is a right hom for $\duof_{v}$, and this leads to
\begin{eqnarray} W\ast (X\circ Y)\cong X\circ (W\ast Y) .\end{eqnarray}
This completes the proof. \qed
\end{proof}

\begin{remark}
Under the condition of Proposition \ref{closureprop}, it follows that the $\duof_{h}$-functors
$$\xymatrix{-\ast X,\quad -\circ X, \quad X\circ -\quad : \duof_{h} \ar[r] & \duof_{h}}$$
all preserve weighted colimits.
\end{remark}

\begin{proposition}\label{closurepropmodified}
For any duoidal $\ev$-category $\duof$, condition $(ii)$ of Proposition \ref{closureprop} is equivalent to
\begin{itemize}
\item[(ii)$'$] there exist $\ev$-natural isomorphisms
\begin{eqnarray}\label{modifiediso}
X\ast (J\circ Y) \cong X\circ Y\cong Y\ast (X\circ J).
\end{eqnarray}
\end{itemize}
\end{proposition}
\begin{proof}
(ii)$\Longrightarrow$(ii)$'$ The second isomorphism of (ii)$'$ comes from the first isomorphism of (ii) with $Y=J$ and $W$ replaced by $Y$.
The first isomorphism of (ii)$'$ comes from the second isomorphism of (ii) with $X=J$ and $W$ replaced by $X$.

\noindent (ii)$'\Longrightarrow$ (ii) Using (ii)$'$, we have 
\begin{eqnarray*}
X\circ (W\ast Y) &\cong & (W\ast Y)\ast (X\circ J) \\
& \cong & W\ast (Y\ast (X\circ J)) \\
& \cong & W\ast (X\circ Y),\textrm{ and} \\
(W\ast X)\circ Y & \cong & (W\ast X)\ast (J\circ Y) \\
& \cong & W\ast (X\ast (J\circ Y)) \\
& \cong & W\ast (X\circ Y). \qed
\end{eqnarray*}
\end{proof}

\begin{theorem}
Suppose $\duof$ is a duoidal $\ev$-category which is horizontally left closed, has equalizers, and satisfies condition (ii)$'$ of Proposition \ref{closurepropmodified}.
Suppose $M$ is a bimonoid in $\duof$ and regard $\duof^{\ast M}$ as a monoidal $\duof_{h}$-category as in Theorem \ref{bimonoidisoclassmonoidal}.
The following conditions are equivalent:
\begin{enumerate}
\item[(i)] $M$ is a (left, right) Hopf bimonoid;
\item[(ii)] $-\ast M$ is a (left, right) Hopf opmonoidal monad on $\duof_{h}$.
\end{enumerate}
If $\duof_{v}$ is a closed monoidal $\ev$-category then these conditions are also equivalent to
\begin{enumerate}
\item[(iii)] $\duof^{\ast M}$ is (left, right) closed and $U_{M}:\duof^{\ast M} \longrightarrow \duof_{h}$ is strong (left, right) closed.
\end{enumerate}
\end{theorem}

\begin{proof}
(ii) $\Longleftrightarrow$ (iii) under the extra condition on $\duof_{v}$ by [BLV] as extended by [CLS]. \\
(ii) $\Longrightarrow$ (i) by taking $X=Y=J$ in (\ref{lefthopfv}), we see that $v_{\ell} (X,Y) = v_{\ell}$. \\
(i) $\Longrightarrow$ (ii) Proposition \ref{closurepropmodified} (ii)$'$ and associativity of $\ast$ yield the isomorphisms
$$X\circ (Y\ast J)\cong Y\ast (X\circ J),$$
$$(Y\ast J)\circ X\cong Y\ast (J\circ X),\textrm{ and}$$
$$(Y\ast J)\ast X\cong Y\ast (J\ast X),$$
showing that $X\circ -$, $-\circ X$ and $-\ast X$ preserve the canonical weighted colimit of Proposition \ref{cornerjisdense} (since colim$(W,S) \cong W\ast S$ when $S: {\cal J}\longrightarrow\duof_{h}$).

Using Proposition \ref{gammaisatwocell}, we see that $v_{\ell} (X,Y)$ is an $\duof_{h}$-natural transformation, in the variables $X$ and $Y$, between two $\duof_{h}$-functors that preserve weighted colimits of the form
$$\textrm{colim}(Z,J)\cong Z\ast J \cong Z.$$
By Proposition \ref{thetacolimisinv}, $v_{\ell}(X,Y)$ is invertible if $v_{\ell}(J,J) = v_{\ell}$ is.\qed
\end{proof}

\begin{example}
Any braided closed monoidal $\ev$-category $\duof$, regarded as duoidal by taking both $\ast$ and $\circ$ to be the monoidal structure given on $\duof$, is an example satisfying the conditions of Proposition \ref{closureprop}.
\end{example}

\begin{remark}
One reading of Proposition \ref{closurepropmodified} (ii)$'$ is that, to know $\circ$ we only need to know $\ast$ and either $J\circ -$ or $-\circ J$.
Proposition \ref{closureprop} (ii) also yields 
\begin{eqnarray}\label{tobededucedfrom}
Y\circ (W\ast \mathbf{1} ) \cong W\ast Y \cong (W\ast \mathbf{1})\circ Y
\end{eqnarray}
showing that to know $\ast$ we only need to know $\circ$ and $-\ast\mathbf{1}$.
From (\ref{modifiediso}) we deduce 
\begin{eqnarray}
\mathbf{1}\ast (J\circ X) \cong X \cong \mathbf{1}\ast (X\circ J)
\end{eqnarray}
and from (\ref{tobededucedfrom}) we deduce
\begin{eqnarray}
J\circ (X\ast\mathbf{1} ) \cong X \cong (X\ast \mathbf{1} ) \circ J
\end{eqnarray}
showing each of the composites
\begin{eqnarray}\label{theequivalences}
\xymatrix{\duof \ar[r]^-{-\circ J} & \duof \ar[r]^-{\mathbf{1}\ast -} & \duof, && \duof \ar[r]^-{J\circ -} & \duof \ar[r]^-{\mathbf{1}\ast -} & \duof, \\ \duof \ar[r]^-{-\ast\mathbf{1}} & \duof \ar[r]^-{J\circ -} & \duof, && \duof \ar[r]^-{-\ast\mathbf{1}} & \duof \ar[r]^-{-\circ J} & \duof }
\end{eqnarray}
to be isomorphic to the identity $\ev$-functor of $\duof$.
From the first and last of these we see that $-\circ J$ is an equivalence and
\begin{eqnarray}
\mathbf{1}\ast - \cong -\ast \mathbf{1}
\end{eqnarray}
both sides being inverse equivalences for $-\circ J$.
From the second of (\ref{theequivalences}) it then follows that $\mathbf{1}\ast -$ is an inverse equivalence for $J\circ -$.
Consequently
\begin{eqnarray}
J\circ - \cong -\circ J.
\end{eqnarray}
\end{remark}

\section{Warped monoidal structures}\label{warpedmonstructsect}

Let $\ba = (\ba , \otimes , I )$ be a monoidal category.
The considerations at the end of Section \ref{hopfbimonoids} suggest the possibility of defining a tensor product on $\ba$ of the form
$$A\warmon B = TA\otimes B$$
for some suitable functor $T:\ba\longrightarrow\ba$.
In the case of Section \ref{hopfbimonoids}, the functor $T$ was actually an equivalence but we will not assume that here in the first instance.

A \emph{warping of $\ba$} consists of the following data:
\begin{itemize}
\item[(a)] a functor $T:\ba\longrightarrow\ba$;
\item[(b)] an object $K$ of $\ba$;
\item[(c)] a natural isomorphism
$$\xymatrix{v_{a,b} : T(TA\otimes B) \ar[r] & TA\otimes TB};$$
\item[(d)] an isomorphism
$$\xymatrix{v_{0} : TK \ar[r] & I};\textrm{ and}$$
\item[(e)] a natural isomorphism
$$\xymatrix{k_{A} : TA\otimes K \ar[r] & A};$$
\end{itemize}
such that the following diagrams commute.
\begin{eqnarray}\label{warpassoc}\xymatrix{
T(TA\otimes B)\otimes TC \ar[rr]^-{v_{a,b}\otimes 1} && (TA\otimes TB)\otimes TC \ar[d]^-{a_{TA,TB,TC}} \\
T(T(TA\otimes B)\otimes C) \ar[u]^-{v_{TA\otimes B , C}} \ar[d]_-{T(v_{a,b}\otimes 1)} && TA\otimes (TB\otimes TC) \\
T((TA\otimes TB)\otimes C) \ar[dr]_-{Ta_{TA,TB,C}\phantom{AA}} && TA\otimes T(TB\otimes C) \ar[u]_-{1\otimes v_{b,c}} \\
& T(TA\otimes (TB\otimes C)) \ar[ru]_-{\phantom{AA}v_{A,TB\otimes C}} &
}\end{eqnarray}
\begin{eqnarray}\label{warpunit}\xymatrix{
T(TA\otimes K) \ar[d]_-{Tk_{A}} \ar[rr]^-{v} && TA\otimes TK \ar[d]^-{1\otimes v_{0}} \\
TA && TA\otimes I \ar[ll]^-{r_{TA}}
}\end{eqnarray}

\begin{remark}
If $T:\ba\longrightarrow\ba$ is essentially surjective on objects and fully-faithful on isomorphisms then all we need to build it up to a warping is $v_{A,B}$ as in (c) satisfying (\ref{warpassoc}).
For $K$ and $v_{0}$ exist by essential surjectivity and $k_{A}$ is defined by (\ref{warpunit}).
\end{remark}

\begin{proposition}\label{liftthisprop}
A warping of $\ba$ determines a monoidal structure on $\ba$ defined by the tensor product 
$$A\warmon B = TA\otimes B$$
with unit object $K$ and coherence isomorphisms
$$\xymatrix{\alpha : T(TA\otimes B)\otimes C \ar[r]^-{v\otimes 1} & (TA\otimes TB)\otimes  C \ar[r]^-{a} & TA\otimes (TB\otimes C)}$$
$$\xymatrix{\ell : TK \ar[r]^-{v_{0}\otimes 1} & I\otimes B \ar[r]^-{\ell} & B}$$
$$\xymatrix{r: TA\otimes K \ar[r]^-{k} & A}.$$
\end{proposition}
\begin{proof}
The pentagon condition for $\warmon$ is obtained from (\ref{warpassoc}) by applying $-\otimes D$.
Similarly, the unit triangle is obtained from (\ref{warpunit}) by applying $-\otimes B$.\qed
\end{proof}

In investigating when $\otimes$ and $\warmon$ together formed a duoidal structure on $\ba$, we realized we could use a lifting of Proposition \ref{liftthisprop} to a monoidal bicategory $\bm$.
We now describe this lifted version.
The duoidal structure formed by $\otimes$ and $\warmon$ will be explained in an example.

A \emph{warping of a monoidale} $A= (A,m,i)$ in a monoidal bicategory $\bm$ consists of
\begin{itemize}
\item[(a)] a morphism $t: A \longrightarrow A$;
\item[(b)] a morphism $k: I\longrightarrow A$;
\item[(c)] an invertible $2$-cell
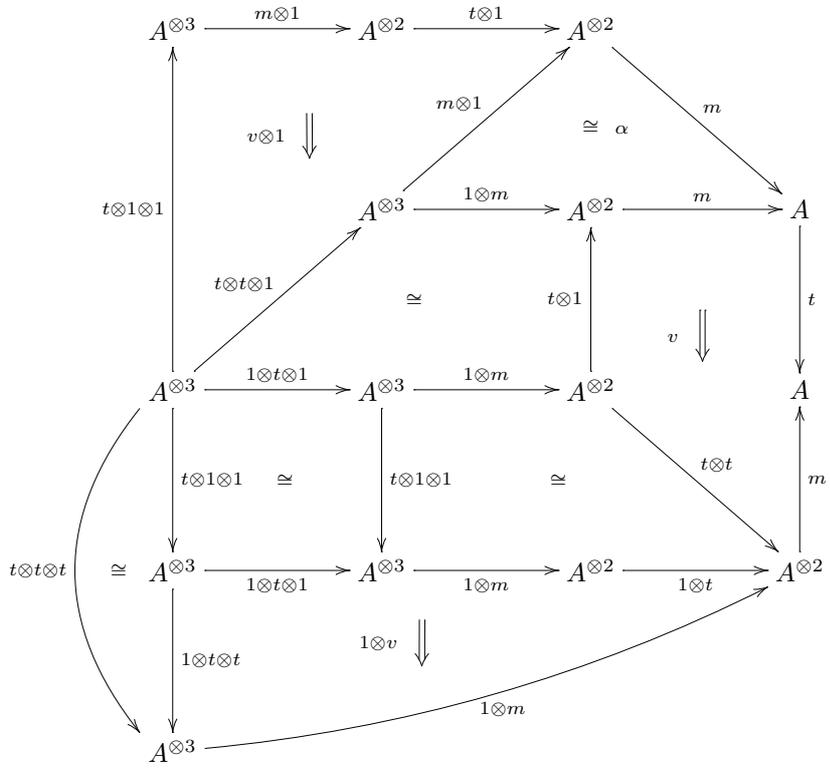
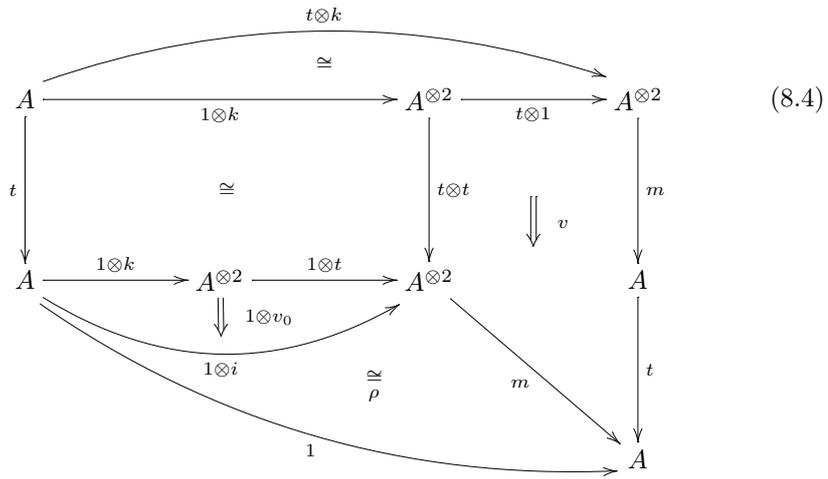
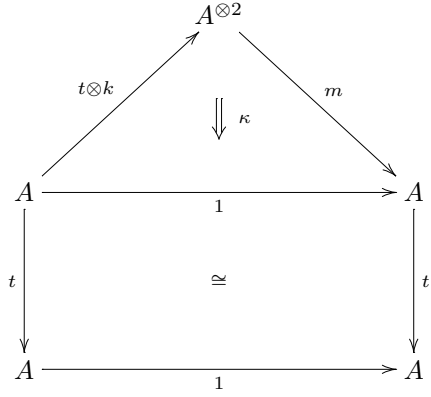
$$\xymatrix{
& A\otimes A \ar[rr]^-{m}_<<<<<<<<<<<{\phantom{A}}="1" && A \ar[dr]^-{\phantom{A}t}& \\ 
A\otimes A \ar[ru]^-{t\otimes 1} \ar[rr]_-{t\otimes t} \ar@{}[rrrr]^<<<<<<<<<<<<<<<<<<<<<<<<<<<<<<{\phantom{A}}="2" \ar@{=>}"1";"2"^-{\phantom{a}v} && A\otimes A \ar[rr]_-{m} && A\,\, ; \\ 
}$$
\item[(d)] an invertible $2$-cell
$$\xymatrix{
\ar@{}[rrrr]_<<<<<<<<<<<<<<<<<<<<<<*!/d1.5pt/{\labelstyle{\phantom{A}}}="1" && A \ar[rrd]^-{t}  && \\
I \ar[rru]^-{k} \ar[rrrr]_-{i}^-{\phantom{a}}="2" \ar@{=>}"1";"2"^-{\phantom{a}v_{0}} &&&& A\,\,;
}$$
\item[(e)] an invertible $2$-cell
$$\xymatrix{
\ar@{}[rrrr]_<<<<<<<<<<<<<<<<<<<<<<*!/d1.5pt/{\labelstyle{\phantom{A}}}="1" && A \ar[rrd]^-{m}  && \\
I \ar[rru]^-{t\otimes k} \ar[rrrr]_-{1}^-{\phantom{a}}="2" \ar@{=>}"1";"2"^-{\phantom{a}\kappa} &&&& A\,\,;
}$$
\end{itemize}
satisfying

\begin{eqnarray}\label{superbigdiag1}
\xymatrix{
&& A^{\otimes 3} \ar[rr]^-{m\otimes 1}_<*!/d12pt/{\labelstyle{\phantom{A}}}="5" && A^{\otimes 2} \ar@/^/[rrdd]^-{t\otimes 1} \ar[dd]^-{t\otimes 1} \ar@/^8ex/@<+0.3ex>[dddd]^-{t\otimes t}^>>>{\phantom{AAa}}="10" && \\
&&&&&& \\
A^{\otimes 3} \ar@/^/[rruu]^-{t\otimes 1\otimes 1} \ar@/_/[rrdd]_-{t\otimes t\otimes t}^>>>>>>*!/u23pt/{\labelstyle{\cong}} \ar[rr]_-{t\otimes t\otimes 1}  && A^{\otimes 3} \ar[rr]^<*!/u12pt/{\labelstyle{\phantom{A}}}="6"^-{m\otimes 1}="3" \ar@{=>}"5";"6"^-{\phantom{A}v\otimes 1} \ar[dd]^-{1\otimes t\otimes t} && A^{\otimes 2} \ar[dd]^-{1\otimes t} \ar@{}[rr]|<<<{\cong} && A^{\otimes 2} \ar[dd]^-{m}_>>{\phantom{AAa}}="11" \ar@{=>}"11";"10"_-*!/u3pt/{\labelstyle{v}} \\
&&&&&& \\
&& A^{\otimes 3} \ar[dd]_-{1\otimes m}^{\phantom{m}}="1" \ar[rr]_{m\otimes 1}="4" \ar@{}"3";"4"|-{\cong} && A^{\otimes 2} \ar[dd]_-{m}="2" \ar@{}"1";"2"|-{\cong}_*!/d2pt/{\labelstyle{\alpha}} && A \ar@/^/[lldd]^-{t} \\
&&&&&& \\
&& A^{\otimes 2} \ar[rr]_-{m} && A &&
} 
\end{eqnarray}

\begin{eqnarray*}=\end{eqnarray*}

\begin{eqnarray*}
\xymatrix{
A^{\otimes 3} \ar[rr]^-{m\otimes 1}_>>>>>>*!/d16pt/{\labelstyle{\phantom{A}}}="1" && A^{\otimes 2} \ar[rr]^-{t\otimes 1} && A^{\otimes 2} \ar[rrdd]^-{m} \ar@{}[dd]|<<<<<<<<<<<{\cong}^<<<<<<<<<<<{\phantom{A}\alpha} && \\
&&&&&& \\
&& A^{\otimes 3} \ar[rruu]^-{m\otimes 1} \ar[rr]^-{1\otimes m} \ar@{}[dd]^-{\phantom{A}\cong} && A^{\otimes 2} \ar[rr]^-{m}_-*!/d20pt/{\labelstyle{\phantom{A}}}="5" && A \ar[dd]^-{t} \\
&&&&&& \\
A^{\otimes 3} \ar[uuuu]^-{t\otimes 1\otimes 1} \ar[rruu]^-{t\otimes t\otimes 1}^>>>>>>>*!/u20pt/{\labelstyle{\phantom{A}}}="2" \ar@{=>}"1";"2"_-{v\otimes 1\phantom{A}} \ar@/_+7ex/@<-1.5ex>[dddd]_-{t\otimes t\otimes t}^-{\phantom{Aa}\cong} \ar[rr]^-{1\otimes t\otimes 1} \ar[dd]^-{t\otimes 1\otimes 1} && A^{\otimes 3} \ar[rr]^-{1\otimes m} \ar[dd]^-{t\otimes 1\otimes 1}_{\cong\phantom{AAAAA}} && A^{\otimes 2} \ar[uu]^-{t\otimes 1} \ar[rrdd]^-{t\otimes t} \ar@{}[dd]_-{\cong\phantom{A}} \ar@{}[rr]^-{\phantom{A}}="6" \ar@{=>}"5";"6"_-{v\phantom{A}} && A \\
&&&&&& \\
A^{\otimes 3} \ar[rr]_-{1\otimes t\otimes 1} \ar[dd]^-{1\otimes t\otimes t} && A^{\otimes 3} \ar[rr]_-{1\otimes m}_<<*!/d7pt/{\labelstyle{\phantom{A}}}="3" && A^{\otimes 2} \ar[rr]_-{1\otimes t} && A^{\otimes 2} \ar[uu]_-{m} \\
&&&&&& \\
A^{\otimes 3} \ar@/_+3ex/[rrrrrruu]_-{1\otimes m}  && \ar@{}[rr]^<<<<<*!/u16pt/{\labelstyle{\phantom{A}}}="4" \ar@{=>}"3";"4"_-{1\otimes v\phantom{A}} &&&&
} 
\end{eqnarray*}

and

\begin{eqnarray}\label{superbigdiag2}
\xymatrix{
A \ar[dd]_-{t}="3" \ar[rrrr]_-{1\otimes k} \ar@/^+5ex/@<+1ex>[rrrrrr]^-{t\otimes k}_-*!/d8pt/{\labelstyle{\cong}} &&&& A^{\otimes 2} \ar[dd]^-{t\otimes t}="4" \ar@{}"3";"4"|-{\cong} \ar[rr]_-{t\otimes 1}_<<<<<<<<<<*!/d19pt/{\labelstyle{\phantom{A}}}="7" && A^{\otimes 2} \ar[dd]^-{m} \\
&&&&&&\\
A \ar@{}[rrrr]_<<<<<<<<<<<<<<<<<<<<<<<(0.005){\phantom{.}}="1" \ar[rr]^-{1\otimes k} \ar@/_+5ex/@<-1ex>[rrrrrrdd]_-{1}  \ar@/_+6ex/@<-0.5ex>[rrrr]_-{1\otimes i}="5"^-{\phantom{.}}="2" \ar@{=>}"1";"2"^-{\phantom{A}1\otimes v_{0}} && A^{\otimes 2} \ar[rr]^-{1\otimes t} && A^{\otimes 2} \ar[rrdd]_-{m}="6" \ar@{}"5";"6"|-{\cong}_-*!/d2pt/{\labelstyle{\rho}} \ar@{}[rr]^<<<<<<<<<<*!/u2pt/{\labelstyle{\phantom{A}}}="8" \ar@{=>}"7";"8"^-{\phantom{A}v} && A \ar[dd]^-{t} \\
&&&&&& \\
&&&&&& A
}
\end{eqnarray}

\begin{eqnarray*}
=
\end{eqnarray*}

\begin{eqnarray*}
\xymatrix{
\ar@{}[rrrr]_-*!/d16pt/{\labelstyle{\phantom{A}}}="1" && A^{\otimes 2} \ar[rrdd]^-{m} && \\
&&&&\\
A \ar[rruu]^-{t\otimes k} \ar[dd]_-{t}="3" \ar[rrrr]_-{1}^-*!/u8pt/{\labelstyle{\phantom{A}}}="2" \ar@{=>}"1";"2"^-{\phantom{a}\kappa} &&&& A \ar[dd]^-{t}="4" \ar@{}"3";"4"|-{\cong} \\
&&&& \\
A \ar[rrrr]_-{1} &&&& A
}
\end{eqnarray*}

\begin{proposition}\label{warpmonoidale}
A warping of a monoidale $A$ determines a monoidale structure on $A$ defined by
$$\xymatrix{_{t}m : A\otimes A \ar[r]^-{t\otimes 1} & A\otimes A \ar[r]^-{m} & A}$$
$$\xymatrix{I \ar[rr]^-{k} && A}$$
$$\xymatrix{
A^{\otimes 3} \ar[dddd]_-{1\otimes t\otimes 1\phantom{A}} \ar[rrrdd]_-{t\otimes t\otimes 1}="1" \ar[rr]^-{t\otimes 1\otimes 1} && A^{\otimes 3} \ar[rr]^-{m\otimes 1}_-*!/d7pt/{\labelstyle{\phantom{A}}}="3" && A^{\otimes 2} \ar[rr]^-{t\otimes 1} && A^{\otimes 2} \ar[dddd]^-{m} \\
&&&&&&\\
&&& A^{\otimes 3} \ar[rrruu]_-{m\otimes 1}="2" \ar@{}"1";"2"_-{\phantom{A}}="4" \ar@{=>}"3";"4"^-{\phantom{A} v\otimes 1} \ar[rdd]_-{1\otimes m} &&& \\
&&&&&&\\
A^{\otimes 3} \ar[rr]_-{1\otimes m} && A^{\otimes 2} \ar[rr]_-{t\otimes 1} \ar@{}[uu]^>>>>>>>{\cong\phantom{AA}} && A^{\otimes 2} \ar[rr]_-{m} \ar@{}[rruuuu]|<<<<<<<<<<<<<<<<<<<<{\cong}_-{\alpha} && A
}$$

$$\xymatrix{
&& A^{\otimes 2} \ar[rr]^-{t\otimes 1}_<{\phantom{a}}="1" && A^{\otimes 2} \ar[rrd]^-{m}="4"&& \\
A \ar[rru]^-{k\otimes 1} \ar@/_4ex/[rrrru]_-{i\otimes 1}="3"^<<<<<<<<<<<<<<(0.01){\phantom{a}}="2" \ar@{=>}"1";"2"_-{v_{0}\otimes 1\phantom{a}} \ar@/_3ex/@<-1ex>[rrrrrr]_-{1}  \ar@{}"3";"4"|-{\phantom{AA}\cong}_-{\phantom{AAA}\lambda}  &&&&&& A
}$$

$$\xymatrix{
A \ar[rr]^-{t\otimes k} \ar@/_5ex/@<-1ex>[rrrr]_-{1}^-*!/d4pt/{\labelstyle{\phantom{a}}}="2" \ar@{}[rrrr]|-*!/u4pt/{\labelstyle{\phantom{A}}}="1" \ar@{=>}"1";"2"^-{\phantom{a}\kappa} && A^{\otimes 2} \ar[rr]^-{m} && A
}$$
\end{proposition}
\begin{proof}
Conditions (\ref{superbigdiag1}) and (\ref{superbigdiag2}) yield the two axioms for a monoidale $(A, {_{t}m},k)$.\qed
\end{proof}

\begin{example}
Suppose $\duof$ is a duoidal $\ev$-category satisfying the second isomorphism of (\ref{tobededucedfrom}).
Define a $\ev$-functor $T:\duof\longrightarrow\duof$ by
$$T = -\ast \mathbf{1}\, .$$
The horizontal right unit isomorphism gives
$$T(J) = J\ast\mathbf{1}\cong\mathbf{1}$$
and (\ref{tobededucedfrom}) gives
\begin{eqnarray*}
T(TA\circ B) & = & ((A\ast\mathbf{1})\circ B)\ast\mathbf{1} \\
& \cong & (A\ast B)\ast\mathbf{1} \\
& \cong & A\ast (B\ast\mathbf{1}) \\
& \cong & (A\ast\mathbf{1} )\circ (B\ast\mathbf{1}) \\
& = & TA\circ TB\, .
\end{eqnarray*}
Finally, we have
\begin{eqnarray*}
TA\circ J & = & (A\ast\mathbf{1} )\circ J \\
& \cong & A\ast J \\
& \cong & A\, .
\end{eqnarray*}
This gives an example of a warping in $\bm = \ev$-Cat of the monoidale (= monoidal $\ev$-category) $\duof_{v}$.
Proposition \ref{warpmonoidale} gives back $\duof_{h}$.
\end{example}

\begin{example}
Consider the case of $\bm=\textrm{Mon}(\ev\textrm{-Cat})$.
A monoidale is a duoidal $\ev$-category $(\duof_{h} , \circ , \mathbf{1})$.
A warping of this monoidale consists of a monoidal $\ev$-functor $T:\duof_{h}\longrightarrow\duof_{h}$, a monoid $K$ in $\duof_{h}$, a horizontally monoidal $\ev$-natural isomorphism $v:T(TA\circ B)\cong TA\circ TB$, a horizontal monoid isomorphism $v_{0} : TK\cong \mathbf{1}$, and a horizontally monoidal $\ev$-natural isomorphism $k:TA\circ K\cong A$, subject to the two conditions.
Proposition \ref{warpmonoidale} gives the recipe for obtaining a duoidal $\ev$-category $(\duof_{h} , (T-)\circ - , K)$.
In particular, take $\ev=\,$Set and consider a lax braided monoidal category $\ba = (\ba , \otimes , I , c)$ as a duoidal category; the lax braiding gives the monoidal structure on $\otimes : \ba\times\ba\longrightarrow\ba$.
A warping consists of a monoidal functor $T:\ba\longrightarrow\ba$, a monoid $K$ in $\ba$, a monoidal natural $v:T(TA\otimes B)\cong TA\otimes TB$, a monoid isomorphism $v_{0} : TK\cong I$, and a monoidal natural $k:TA\otimes K\cong A$, satisfying the conditions (\ref{warpassoc}) and (\ref{warpunit}).
Proposition \ref{warpmonoidale} then shows that the recipe of Proposition \ref{liftthisprop} yields a duoidal category $(\ba ,\otimes , I, \warmon , K)$.
\end{example}

\begin{center}
--------------------------------------------------------
\end{center}

\appendix

\end{document}